\documentclass[12 pt]{amsart}
\usepackage{thmtools}
\usepackage{
	amsmath,  amssymb,  amsthm,   amscd,
	gensymb,  graphicx, etoolbox, booktabs,
	stackrel, mathtools    
}
\usepackage[usenames,dvipsnames]{xcolor}
\usepackage{hyperref}
\usepackage{mathrsfs}
\usepackage{graphicx}
\usepackage{enumitem}
\usepackage{romannum}
\usepackage{mathtools}
\usepackage[utf8]{inputenc}
\usepackage{listings}
\usepackage[capitalise]{cleveref}
\usepackage{placeins}
\usepackage{lipsum,eso-pic,xcolor}
\usepackage{lineno}

\usepackage{comment}

\usepackage{tikz}
\usetikzlibrary{arrows,decorations.pathmorphing,backgrounds,positioning,fit}
\usetikzlibrary{positioning}
\hypersetup{
    colorlinks=true,
    linkcolor=blue,
    hypertexnames=false,
}
\usetikzlibrary{trees}
\usetikzlibrary{arrows}

\def\supp{\operatorname{supp}}

\def\reg{\operatorname{reg}}
\def\deg{\operatorname{deg}}

\def\max{\operatorname{max}}

\DeclareMathOperator{\Ass}{Ass}
\def\inn{\operatorname{in_<}}

\newcommand{\leaf}{\textup{Leaf}}
\newcommand{\vpi}{\textup{v}_{\pp^{+}(G)}(\II_{G})}

\newcommand{\size}[1]{\left| #1 \right|} 

\newcommand{\ZZ}{\mathbb{Z}}

\newcommand{\II}{\mathcal{I}}
\newcommand{\JJ}{\mathcal{J}}
\newcommand{\dng}{\mathcal{D}_{n}(G)}

\newcommand{\iv}{\textup{iv}}
\newcommand{\mm}{\mathfrak{m}}
\newcommand{\pp}{\mathfrak{p}}
\newcommand{\qq}{\mathfrak{q}}
\newcommand{\RR}{\mathbb{R}}

\newcommand{\bI}{\mathbb{I}}
\newcommand{\VV}{\mathbb{V}}

\newcommand{\vpp}{\vn_{\pp}(I)}

\newcommand{\vd}{v_{d}(G)}
\newcommand{\+}{\pp^{+}(G)}
\newcommand{\GG}{\mathcal{G}}

\newcommand{\jkm}{J_{K_{m},G}}
\newcommand{\MG}{\mathcal{M}(G)}
\newcommand{\TT}{\mathcal{T}}

\newcommand{\pe}{\pp_{\emptyset}(G)}
\newcommand{\vp}{\textup{v}_{\pp_{\emptyset}(G)}(L_G^{\mathbb{K}}(d))}

\newcommand{\KK}{\mathbb{K}}
\newcommand{\lcm}{\mathrm{lcm}}

\newcommand{\OR}{\textup{OR}_{d}(G)}
\newcommand{\vn}{\mathrm{v}}
\newcommand{\dcng}{\mathcal{D}_{c,n}(G)}
\newcommand\restr[2]{{
   \left.\kern-\nulldelimiterspace 
   #1 
   \vphantom{\small|} 
   \right|_{#2} 
   }}

\newcommand{\vf}{\mathbf{v}}
\newcommand{\ini}{\textup{in}_{\prec}}
\newcommand{\vg}{V_{\geq d}(G)}
\newcommand{\vlg}{\mathrm{v}_{I_{K_{n}}}(L_{G}^{\mathbb{R}}(2))}
\newcommand*\closure[1]{\overline{#1}}

\newtheorem{proposition}{Proposition}[section]
\newtheorem{lemma}[proposition]{Lemma}
\newtheorem{corollary}[proposition]{Corollary}
\newtheorem{theorem}[proposition]{Theorem}
\newtheorem{definition}[proposition]{Definition}
\newtheorem{remark}[proposition]{Remark}
\newtheorem{notation}[proposition]{Notation}
\newtheorem{example}[proposition]{Example}
\newtheorem{question}[theorem]{Question}

\advance\headheight1.15pt

\newtheorem{innercustomthm}{Theorem}
\newenvironment{customthm}[1]
  {\renewcommand\theinnercustomthm{#1}\innercustomthm\itshape}
  {\endinnercustomthm}

\begin{document}

	\pagenumbering{arabic}
	
\title[]{V-number of Lovász--Saks--Schrijver Ideals and (parity) binomial edge ideals of graphs} 
	
	\author[Manohar Kumar]{Manohar Kumar}
	\address{Department of Mathematics, Indian Institute of Technology Madras, Chennai, INDIA - 600036.}
	\email{manhar349@gmail.com}

	\author[Emiliano Liwski]{Emiliano Liwski}
	\address{Department of Mathematics, KU Leuven, Celestijnenlaan 200B, 3001 Leuven, Belgium.}
	\email{emiliano.liwski@kuleuven.be}
	
	
	\thanks{AMS Classification 2020: 13F20, 05E40, 05E99, 13D45}

	\maketitle
	
	\begin{abstract}
In this paper, we introduce a new framework for computing certain localized $\vn$-numbers of a class of ideals called coordinate-saturated ideals, which includes certain classes of Lovász--Saks--Schrijver (LSS) ideals and (generalized)
binomial edge ideals associated with graphs. For a forest graph $G$, we derive an explicit formula for the localized $\vn$-number of the LSS ideal $L_G^{\mathbb{K}}(d)$, denoted by $\vp$, for all $d \geq 2$, where $\mathbb{K}$ is an algebraically closed field. As a consequence, we prove that
$\vn(L_G^{\mathbb{K}}(d)) \leq \reg(R/L_G^{\mathbb{K}}(d)),$
where $\vn(L_G^{\mathbb{K}}(d))$ and $\reg(R/L_G^{\mathbb{K}}(d))$ denote the $\vn$-number of $L_G^{\mathbb{K}}(d)$ and the Castelnuovo--Mumford regularity of $R/L_G^{\mathbb{K}}(d)$, respectively. Also, we give an upper bound for $\vlg$, where $\mathbb{R}$ is field of real numbers. Furthermore, we provide combinatorial descriptions of the localized $\vn$-number of parity binomial edge ideals, denoted by $\vpi$, and, as an application, show that $\vn(\II_G) \leq \reg(R/\II_G)$ for several classes of non-bipartite graphs. Finally, we prove that $\vn(J_G^k)\leq \reg(R/J_G^k)$ for all powers of binomial edge ideals of closed graphs $G$.
\end{abstract}

\section{Introduction}
Let $R = \mathbb{K}[x_1, \ldots, x_n] = \bigoplus_{d = 0}^{\infty} R_d$ be a standard graded polynomial ring over a field $\mathbb{K}$, and let $I \subset R$ be a graded ideal. It is well-known that for every $\mathfrak{p} \in \Ass(I)$, there exists a homogeneous element $f \in R$ such that $\mathfrak{p} = I : f.$ The \textit{$\vn$-number} of $I$ is defined by
\[\vn(I) = \min \left\{ d \ge 0 : \exists \, f \in R_d \text{ and } \mathfrak{p} \in \Ass(I) \text{ such that } I : f = \mathfrak{p} \right\}.\]
Moreover, for each $\mathfrak{p} \in \Ass(I)$, the \textit{local $\vn$-number} of $I$ at $\mathfrak{p}$ is defined by
\[\vn_{\mathfrak{p}}(I) = \min \left\{ d \ge 0 : \exists\, f \in R_d \text{ such that } I : f = \mathfrak{p} \right\}.\]
The $\vn$-number was introduced in \cite{cstpv20} in the context of projective Reed--Muller-type codes. It is also related to classical geometric invariants; for instance, the degree of a truncator of a finite set of projective points \cite{gkr93} can be interpreted as a local $\vn$-number. For developments in this direction, we refer the reader to \cite{bms24,c24,fs25,JARAMILLO2021,jls26}.

The goal of this paper is to study the $\vn$-number of Lov\'{a}sz--Saks--Schrijver ideals, parity binomial edge ideals, and generalized binomial edge ideals. In \Cref{sec:prelim}, we recall the necessary definitions, basic notions, used throughout the paper. In \Cref{new framework}, we introduce a new framework to compute localized $\vn$-number, introducing the notion of {\em coordinate-saturated ideals}; see \Cref{def: coordinate-saturated}. We prove the following theorem; for the definitions of $M$ and $\gamma(I)$ see \Cref{ideal M and gamma}.

\begin{customthm}{\ref{thm: main theorem coordinate-saturated ideals}}
Let $I\subset \KK[x_1,\ldots,x_n]$ be a coordinate-saturated ideal, and suppose that $I$ has a unique minimal prime containing no indeterminates. Then this prime is necessarily
 \[
\pp=\Big(I:\prod_{i\in [n]}x_i \Big).
\]
Moreover, $\vpp\leq \gamma(I),$ with equality whenever $I+M$ is radical. In particular, the equality holds whenever $I$ is binomial.
\end{customthm}

We show that coordinate-saturated ideals encompass the Lovász--Saks--Schrijver ideals of forests over algebraically closed fields and of arbitrary graphs over $\RR$, as well as 
generalized binomial edge ideals of graphs. As an application, we compute certain localized $\vn$-numbers for these classes of ideals.

Let $G$ be a simple graph on the vertex set $V(G)=[n]:=\{1,\ldots,n\}$ and edge set $E(G)$. Let $d$ be a positive integer. Consider the $n\times d$ matrix of indeterminates $X=(x_{i,j})$, and let $R=\KK[x_{ij}: i \in [n], j \in [d]]$ denote the corresponding polynomial ring. For each edge $e=\{i,j\}\in E(G)$, define $f_{e}^{(d)} := \sum_{k=1}^{d} x_{i,k} x_{j,k}.$ The \emph{Lov\'{a}sz--Saks--Schrijver ideal} (or \emph{LSS ideal}) of $G$ in dimension $d$ with respect to $\KK$ is
\begin{equation}\label{fe}
L_{G}^{\KK}(d) = (f_{e}^{(d)}: e\in E(G))\subset R=\KK[x_{ij}: i \in [n], j \in [d]].
\end{equation}
We denote by $\OR := \mathbb{V}(L_G^{\KK}(d)) \subseteq \KK^{nd}$ the variety associated with $L_G^{\KK}(d)$. Its points are precisely the $n$-tuples $(v_1,\ldots,v_n)\in (\KK^d)^n$ satisfying the orthogonality relations $v_i\cdot v_j = 0$ for every edge $\{i,j\}\in E$. We refer the reader to \cite{cw19, hmsw15} for a detailed study of LSS ideals. Herzog et al.~\cite{hmsw15} proved that if $\operatorname{char}(\KK)\neq 2$, then $L_G^{\KK}(d)$ is radical. They also determined the primary decomposition of $L_G^{\KK}(d)$ when $\sqrt{-1}\notin \KK$ and $\operatorname{char}(\KK)\neq 2$. In \Cref{LSS ideals-forest}, we prove the following  theorem:
\begin{customthm}{\ref{thm:local v-number LSS ideals}}
Let $G$ be a forest graph, and let $d \geq 3$. Then the localized $\vn$-number $\vp$ coincides with 
the number of vertices of $G$ of degree at least $d$.
\end{customthm}
Using \Cref{thm:local v-number LSS ideals}, for any forest graphs we show the following theorem:
\begin{customthm}{\ref{thm:v-number-regularity}}
Let $G$ be a forest graph, and let $d \geq 1$. Then $\vn(L_G^{\mathbb{K}}(d)) \leq \reg(R/L_G^{\mathbb{K}}(d))$.
\end{customthm}

In \Cref{LSS ideals real field}, we establish an upper bound for $\vlg$ in the following theorem. 
\begin{customthm}{\ref{LSS-over-real}}
Let $G$ be a simple graph on $[n]$. Then \[
\begin{aligned}
\vlg \leq \gamma_{c,n}(G)
=
\min\{\,&|T| :\; T\subseteq [n]\text{ is dominant,}\\
& G[T]\text{ is connected and non-bipartite}\,\}.
\end{aligned}
\]
\end{customthm}
Another important family of ideals is the \emph{parity binomial edge ideal} of $G$, introduced by Kahle et al.~\cite{kst16}, which is defined as
\[
\II_G=\big(x_ix_j-y_iy_j:\{i,j\}\in E(G)\big)\subseteq R=\KK[x_1,\ldots,x_n,y_1,\ldots,y_n].
\]
Several algebraic properties of parity binomial edge ideals, including primary decomposition, mesoprimary decomposition, Markov bases, and radicality, were studied in \cite{kst16}. Furthermore, in \cite{cj26}, the authors investigated the $\vn$-number of permanental ideals. Herzog et al. \cite{hmsw15} introduced the notion of permanental edge ideals associated to arbitrary graphs,
\[
\Pi_G = \big(x_i y_j + x_j y_i : \{i,j\} \in E(G)\big) \subseteq \KK[x_1, \dots, x_n, y_1, \dots, y_n].
\]
Observe that, for the complete graph $K_n$, the permanental edge ideal
\[
\Pi_{K_n} = \big(x_i y_j + x_j y_i : 1 \le i < j \le n\big)
\]
coincides with the ideal generated by all $2 \times 2$ permanents of the generic $2 \times n$ matrix. Also, when $\operatorname{char}(\KK) \neq 2$, the ideal $\Pi_G$ is essentially equivalent to the parity binomial edge ideal $\II_G$ (see \cite[Remark 3.3]{zbMATH07352276}). Consequently, $\vn(\Pi_G) = \vn(\II_G).$ Bolognini et al.~\cite[Corollary 6.2]{bms18} proved that if $G$ is bipartite, then the parity binomial edge ideal $\II_G$ is essentially equivalent to $L_G^{\mathbb{K}}(2)$. This connection naturally motivates the study of the $\vn$-number and other algebraic invariants of parity binomial edge ideals. In \Cref{parity binomial edge ideals}, we prove the following result for $G$ a simple graph which is connected and non-bipartite; see \Cref{def: q-} for the definition of $\nu_G$.

\begin{customthm}{\ref{thm: vpi equals nuG}}
The localized $\vn$-number $\vpi$ of $\II_{G}$ at the minimal prime $\+$ coincides with the combinatorial invariant $\nu_{G}$. In other words, $\vpi=\nu_{G}.$ 
\end{customthm}

In \Cref{explicit formula for parity binomial edge ideals}, we provide an explicit formula for $\vpi$ in the following theorem:

\begin{customthm}{\ref{thm: vpi equals yng}}
Let $G$ be a non-bipartite simple graph. Then
\[\vpi=\gamma_{n}(G)=\min \{\size{S}: \text{$S$ is a dominating set and $G[S]$ is totally non-bipartite}\},\]
where a graph is totally non-bipartite if all its connected components are non-bipartite.
\end{customthm}
 Furthermore in \Cref{v-number versus regularity for parity binomial edge ideals}, we introduce a class of graphs called decorated tree; (see Definition~\ref{def: decorated tree}) and using \Cref{thm: vpi equals yng} we prove the following theorem:
\begin{customthm}{\ref{thm:decorated-tree}}
Let $G$ be a decorated tree. Then $\vn(\II_G) \leq \reg(R/\II_G).$
\end{customthm}
 
Also, we prove that $\vn(\II_G) \leq \reg(R/\II_G)$ for several classes of connected non-bipartite graphs; see Propositions \ref{prop:path-triangle}, \ref{prop:chord-even-cycle}, \ref{prop:chord-odd-cycle}.

We also focus on the family of generalized binomial edge ideals, introduced by Rauh in~\cite{Rauh13}. Let $m,n\geq 2$, and let $R=\KK[x_{i,j}:i\in [m],,j\in [n]].$
The \emph{generalized binomial edge ideal} of the pair $(K_m,G)$ is defined by
\begin{equation}\label{def: generalized binomial edge ideals}
  J_{K_m,G}=(x_{i,k}x_{j,\ell}-x_{i,\ell}x_{j,k}
: i,j\in[m],\ \{k,\ell\}\in E(G)).  
\end{equation}

Observe that when $m=2$, the generalized binomial edge ideal $J_{K_2,G}$ coincides with the binomial edge ideal of $G$, denoted by $J_G$. Indeed, after identifying $x_{1,j}=x_j$ and $x_{2,j}=y_j$ for each $j\in[n]$, we obtain
\[
J_{K_2,G}= \big(x_i y_j-x_j y_i : i<j,\ {i,j}\in E(G)\big) =: J_G
\subseteq
R=\KK[x_1,\dots,x_n,y_1,\dots,y_n],
\]
which is precisely the \emph{binomial edge ideal} introduced independently by Herzog et al.~\cite{herzog2010binomial} and Ohtani~\cite{ohtani2011graphs}. Recall from the above discussion that Bolognini et al.~\cite[Corollary 6.2]{bms18} proved that, if $G$ is bipartite, then the parity binomial edge ideal $\II_G$ is essentially equivalent to $L_G^{\mathbb{K}}(2)$. Moreover, it is also equivalent to the binomial edge ideal of $G$. In view of this equivalence, it is worth noting that the $\vn$-number of binomial edge ideals has been studied extensively; see \cite{ass24, dey2024v, jaramillo2024connected, liwski2025v}. In \Cref{generalized binomial edge ideals}, we prove that  
\[\vn(J_{G})\geq \vn(J_{K_{3},G})\geq \vn(J_{K_4, G})\geq \cdots \cdots; \text{ see Theorem }~\ref{thm: monotonicity}.\]

In \cite[Theorem 3.3]{sz26}, authors compute localized $\vn$-number of generalized binomial edge ideals. In \Cref{thm:generalized v-number}, we give a different proof for \cite[Theorem 3.3]{sz26} by showing generalized binomial edge ideals are coordinate-saturated ideals. Finally in \Cref{closed graphs}, we prove that for any closed graphs, $\vn(J_G^k) \leq \reg(R/J_G^k)$ for all $k \geq 1$; see \Cref{proppowers}. %

\section{Preliminaries}\label{sec:prelim}

In this section, we recall the necessary preliminaries and notation used throughout the paper.

Throughout the paper, let $G$ denote a connected simple graph with vertex set $V(G)=[n]$ and edge set $E(G)$. For a vertex $x\in V(G)$, the \textit{neighbourhood} of $x$ in $G$, denoted by $N_{G}(x)$, is defined as $N_{G}(x)= \{y \in V(G) \mid \{x,y\} \in E(G)\}$. For a subset $A\subseteq V(G)$, we denote the \textit{induced} subgraph of $G$ on $A$ by $G[A]$, in particular, $G[A]$ is a graph with $V(G[A])=A$ and $E(G[A])=\{e\in E(G)\mid e\subseteq A\}$. Again, by $G\setminus A$, we mean the induced subgraph $G[V(G)\setminus A]$, which we will also often denote as $G_A$. For simplicity of notation, we write $G\setminus v$ to denote the graph $G\setminus \{v\}$ for any vertex $v \in V(G)$. 

A cycle of length \(n\), denoted by $C_n$, is a graph with vertex set
$V(C_n)=\{v_1,\ldots,v_n\}$ and edge set $E(C_n)
=\bigl\{\{v_i,v_{i+1}\}:1\le i\le n-1\bigr\}\cup \bigl\{\{v_1,v_n\}\bigr\}.$
A graph \(G\) is said to be \emph{complete} if every pair of distinct vertices of \(G\) is joined by an edge. The complete graph on \(n\) vertices is denoted by \(K_n\).

The \emph{Castelnuovo--Mumford regularity} (in short, \emph{regularity}) of $R/I$ is defined by
\[
\reg(R/I) = \max \{j - i : \beta_{i, j}(R/I) \ne 0\},
\]
where $\beta_{i, j}(I)$ denotes the $(i, j)$-th graded Betti number of a graded ideal $I \subset R=\mathbb{K}[x_1,\ldots,x_n]$.

\section{A new framework for computing localized v-numbers}\label{new framework}

In this section, we introduce a new method for computing localized v-numbers of a class of ideals, which we call {\em coordinate-saturated ideals}. In the subsequent sections, we apply the results of this section and compute some specific localized v-numbers for some classic families of ideals, such as LSS ideals, parity binomial edge ideals and generalized binomial edge ideals. 
We will first prove some auxiliary results which will be useful throughout the paper. We denote $\KK[\mathbf{x}]=\KK[x_{1},\ldots,x_{n}]$.


\begin{proposition}\label{prop: I+M is radical}
Let $I,M\subseteq \mathbb{K}[\mathbf{x}]$ be radical ideals, where $I$ is
binomial and $M$ is monomial. Assume that every monomial contained in $I+M$ belongs to $M$. Then $I+M$ is radical.
\end{proposition}

\begin{proof}
By \textup{\cite[Proposition~3.4]{eisenbud1996binomial}}, we know that there exists a monomial ideal $M_{1}$ such that
\begin{equation*}
\sqrt{I+M}=\sqrt{I}+M_{1}.
\end{equation*}
We claim that $M_{1}\subseteq M$. Let $u$ be any monomial in $M_{1}.$ Then it follows that $u\in \sqrt{I+M}$ as $\sqrt{I+M}=\sqrt{I}+M_{1}.$ Hence there exists some positive integer $k$ such that $u^{k}\in I+M$. Since $u^{k}$ is also a monomial, and because of the assumption that any monomial in $I+M$ belongs to $M$, it follows that $u^{k}\in M$. Thus, since $M$ is radical, this implies that $u\in M$. Hence $M_{1}\subseteq M$. Combining this with the radicality of $I$, it follows that
\[\sqrt{I+M}=\sqrt{I}+M_{1}=I+M_{1}\subseteq I+M,\]
hence $I+M$ is radical.
\end{proof}

\begin{proposition}\label{prop: ini(I+M)}
Let $I,M\subseteq \mathbb{K}[\mathbf{x}]$ be ideals, where $I$ is
binomial and $M$ is monomial. Assume that every monomial contained in $I+M$ belongs to $M$. Then $\ini(I+M)=\ini(I)+M$, for any monomial order $\prec$.
\end{proposition}

\begin{proof}
Let $\mathcal{G}$ be a homogeneous Gr\"obner basis of $I$ consisting of binomials, which exists by \textup{\cite[Proposition~1.1]{eisenbud1996binomial}}, and let $\mathcal{G}_{1}$ be the minimal generating set of $M$. We claim that $\mathcal{G}\cup \GG_{1}$ forms a Gr\"obner basis of $I+M$. To show this, we use Buchberger’s criterion and verify that every $S$-pair $S(f,h)$ with $f,h \in \mathcal{G}\cup \GG_{1}$ reduces to zero. This is immediate when both $f$ and $h$ belong to $\GG$, or when both belong to $\GG_{1}$, since $\GG$ and $\GG_{1}$ are Gröbner bases of $I$ and $M$ respectively. 

Now, suppose that $f\in \GG$ and $h\in \GG_{1}$. We then have 
\begin{align*}
S(f,h)=\frac{\lcm(\ini(f),h)f}{\ini(f)}-\frac{\lcm(\ini(f),h)h}{h}
=\frac{\lcm(\ini(f),h)(f-\ini(f))}{\ini(f)}.
\end{align*}
Since $f$ is a binomial, it follows that $S(f,h)$ is a monomial. Moreover, by construction, we have $S(f,h)\in I+M$, and hence $S(f,h)\in M$, since every monomial in $I+M$ must belong to $M$ by assumption.
Hence, $S(f,h)\in M$ and it reduces to zero with respect to $\GG_{1}$. This proves that $\GG\cup \GG_{1}$ is a Gr\"obner basis, and hence $\ini(I+M)=\ini(I)+M$.
\end{proof}

Let $\phi:\KK[\mathbf{x}]\rightarrow \KK[\mathbf{y}]$ be a monomial map given by $\phi(x_i)=y^{a_{i}}$ for $i\in [n]$, where $a_i\in \ZZ^{m}$. This induces a $\ZZ^{m}$-grading given by 
\begin{equation}\label{grading phi}\deg(x_i)=a_{i}\in \ZZ^{m}.
\end{equation}
Note that if $\phi(x^{r})=y^{s}$, then $\deg(x^{r})=s$.

\begin{proposition}\label{prop: sum of initials and upper closed}
Let $I,M\subseteq \KK[\mathbf{x}]$ be ideals, with $M$ monomial, and let $\phi:\KK[\mathbf{x}]\rightarrow \KK[\mathbf{y}]$ be a monomial map such that $I$ is homogeneous with respect to the grading in Equation ~\eqref{grading phi}. Suppose that whenever $u$ and $v$ are monomials with $u\in M$ and $\deg(u)\leq \deg(v)$, then $v\in M$. Then $\ini(I+M)=\ini(I)+M$, for any monomial order $\prec$. 
\end{proposition}

\begin{proof}
Let $\mathcal{G}$ be a Gr\"obner basis of $I$ which is graded with respect to the multigrading in Equation~\eqref{grading phi}, and let $\GG_{1}$ be
the minimal generating set of $M$. We claim that $\mathcal{G}\cup \GG_{1}$ forms a Gr\"obner basis of $I+M$.  
To show this, we use Buchberger’s criterion and verify that every $S$-pair $S(f,h)$ with $f,h \in \mathcal{G}\cup \GG_{1}$ reduces to zero. As in the proof of Proposition~\ref{prop: ini(I+M)}, it suffices to consider the case where $f\in \GG$ and $h\in \GG_{1}$. In this case, we have 
\begin{align*}
S(f,h)=\frac{\lcm(\ini(f),h)f}{\ini(f)}-\frac{\lcm(\ini(f),h)h}{h}
=\frac{\lcm(\ini(f),h)(f-\ini(f))}{\ini(f)}.
\end{align*}
Since $f$ is homogeneous with respect to the multigrading in~\Cref{grading phi}, it follows that $\deg(f-\ini(f))=\deg(\ini(f))$, and hence
\begin{align*}
\deg(S(f,h))&=\deg(\lcm(\ini(f),h))+\deg(f-\ini(f))-\deg(\ini(f))\\
&=\deg(\lcm(\ini(f),h))\geq \deg(h).
\end{align*}
Since $h$ is a monomial of $M$, and every monomial of $S(f,h)$ has degree greater or equal that $\deg(h)$, it follows from our assumptions that each such monomial belongs to $M$ and hence $S(f,h)\in M$. Thus, $S(f,h)$ reduces to zero with respect to $\GG_{1}$. This proves that $\GG\cup \GG_{1}$ is a Gr\"obner basis, and hence 
$\ini(I+M)=\ini(I)+M$.
\end{proof}

We now provide a sufficient condition for having the property that every monomial of $I+M$ belongs to $M$.

\begin{proposition}\label{prop: sufficiento condition for monomials in I+M}
Let $I,M\subseteq \KK[\mathbf{x}]$ be ideals, with $M$ monomial, and let $\phi:\KK[\mathbf{x}]\rightarrow \KK[\mathbf{y}]$ be a monomial map such that $I\subseteq \ker(\phi)$ and with the property that whenever $u$ and $v$ are monomials with $u\in M$ and $\phi(u)= \phi(v)$, then $v\in M$. 
Then every monomial contained in $I+M$ belongs to $M$.
\end{proposition}

\begin{proof}
Suppose that we have
\begin{equation}\label{upq2}u=p+q \quad \text{with  $p\in I,q\in M$ and $u$ a monomial}.\end{equation}
We must show that $u\in M$. Since $I\subseteq \ker(\phi)$, applying $\phi$ to~\Cref{upq2} gives $\phi(u)=\phi(q)$. Set
\[q:=\sum_{r\in \ZZ_{\geq 0}^{n}}c_{r}x^{r}.\]
Then, we have
\[\phi(q)=\sum_{r\in \ZZ_{\geq 0}^{n}}c_r\phi(x^{r}).\]
Since $\phi(q)=\phi(u)$, there must exist at least one $r\in \ZZ_{\geq 0}^{n}$ such that $c_{r}\neq 0$ and $\phi(x^{r})=\phi(u)$. 
Since $q\in M$, it follows that $x^{r}\in M$ as well. Since $x^{r}\in M$ and $\phi(x^{r})=\phi(u)$, it follows from our assumption that $u\in M$, completing the proof.
\end{proof}

\subsection{v-number for coordinate-saturated ideals}\label{framework coordinate saturated}

Let $I\subset \KK[\mathbf{x}]=\KK[x_{1},\ldots,x_n]$ be a radical ideal. For each $S\subseteq [n]$, we define
\[P_{S}(I):=\Big((I+(x_{i}:i\in S)): \prod_{i\notin S}x_{i}^{\infty}\Big).\]
\begin{definition}\label{def: coordinate-saturated}\normalfont
A radical ideal $I\subseteq \KK[x_1,\ldots,x_n]$ is said to be \emph{coordinate-saturated} if every minimal prime of $I$ is of the form $P_S(I)$ for some $S\subseteq [n]$, and, for every such $S$,
\begin{equation}\label{second cond coord saturated}
(I+(x_i:i\in S))\cap \KK[x_i:i\notin S]
=
I\cap \KK[x_i:i\notin S].
\end{equation}
\end{definition}

\begin{lemma}\label{expression of PSI}
Let $I$ be a coordinate-saturated ideal and let $P_{S}(I)$ be a minimal prime of $I$. Then
\begin{equation*}
P_{S}(I)=\Big((I\cap \KK[x_i:i\notin S]): \prod_{i\notin S}x_{i} \Big)+(x_{i}:i\in S).
\end{equation*}
\end{lemma}

\begin{proof}
The inclusion $P_{S}(I) \supseteq ((I\cap \KK[x_i:i\notin S]): \prod_{i\notin S}x_{i})+(x_{i}:i\in S)$ is clear. To prove the other inclusion, consider $f\in P_{S}(I)$. By definition, there exists a monomial $u\in (x_{i}:i\notin S)$ such that $fu\in I+(x_{i}:i\in S)$. Decompose $f$ as $f=f_{1}+f_{2}$, where $f_{1}\in \mathbb{K}[x_{i}:i\notin S]$ and $f_{2}\in (x_{i}:i\in S)$. It then follows that $f_{1}u\in  I+(x_{i}:i\in S)$. Note that $f_{1}u\in (I+(x_{i}:i\in S))\cap \KK[x_i:i\notin S]$, and hence by~\Cref{second cond coord saturated}, it follows that $f_{1}u\in I\cap \KK[x_i:i\notin S]$. Consequently, 
\[f_{1}\in \Big((I\cap \KK[x_i:i\notin S]):\prod_{i\notin S}x_{i}^{\infty}\Big)=\Big((I\cap \KK[x_i:i\notin S]):\prod_{i\notin S}x_{i}\Big),\]
where the equality holds since $I$ is radical. It then follows that 
\[f=f_{1}+f_{2}\in \Big((I\cap \KK[x_i:i\notin S]): \prod_{i\notin S}x_{i}\Big)+(x_{i}:i\in S),\]
proves that $P_{S}(I) \subseteq ((I\cap \KK[x_i:i\notin S]): \prod_{i\notin S}x_{i})+(x_{i}:i\in S)$.
\end{proof}

The goal of this section is to study the $\vn$-number of coordinate-saturated ideals. We begin by observing that this class includes several important families of ideals that have been extensively studied in the literature, including binomial edge ideals \cite{herzog2010binomial}, generalized binomial edge ideals \cite{Rauh13}, 
LSS ideals over $\RR$ \cite{hmsw15}, and LSS ideals of forest graphs \cite{liwski2025lov}. In the subsequent sections, we will show that these classes of ideals are coordinate-saturated and apply the results developed here to each of these families. We now provide a sufficient condition for an ideal to be coordinate-saturated.

\begin{notation}
For each $\qq\in \min(I)$, we denote $S(\qq)=\{i\in [n]: x_{i}\in \qq\}$. For each $S\subset [n]$, we define the $\KK$-algebra homomorphism $\pi_{S}:\KK[\mathbf{x}]\rightarrow \KK[\mathbf{x}]$ given by 
\[\pi_{S}(x_{i}):=
\begin{cases}
x_{i} & \text{if $i\notin S$;}\\
0 & \text{if $i\in S$}.
\end{cases}
\]
\end{notation}

\begin{proposition}\label{prop: sufficient conditions for coordinate-saturated}
Let $I$ be a radical ideal. Suppose that, for every $\qq\in\min(I)$, the following conditions hold: \begin{enumerate}[label=\rm(\roman*)] \item\label{s1} $\pi_{S(\qq)}(I)\subseteq I$. \item\label{s2} The ideal $I+(x_i:i\in S(\qq))$ has a unique minimal prime containing no indeterminates $\{x_i:i\notin S(\qq)\}$. \end{enumerate} 
Then $I$ is coordinate-saturated. \end{proposition}

\begin{proof}
We first show that 
\begin{equation}\label{Sq}(I+(x_i:i\in S(\qq)))\cap \KK[x_i:i\notin S(\qq)]
=
I\cap \KK[x_i:i\notin S(\qq)]\end{equation}
for every $\qq\in \min(I)$. The inclusion $\supset$ in~\Cref{Sq} is clear. To show the other inclusion, let $f$ be a polynomial on the left-hand side of~\Cref{Sq}. Since $\pi_{S(\qq)}|_{\KK[x_{i}:i\notin S(\qq)]}=\text{Id}$, we have $\pi_{S(\qq)}(f)=f$. Moreover, since $\pi_{S(\qq)}|_{\KK[x_{i}:i\in S(\qq)]}=0$, it follows that $\pi_{S(\qq)}(f)\in \pi_{S(\qq)}(I)\subseteq I$ by~\Cref{s1}. Hence $f\in I$, which proves the inclusion $\subset$ in~\Cref{Sq}.

By~\Cref{s2}, $I+(x_i:i\in S(\qq))$ has a unique minimal prime containing no indeterminates from $\{x_{i}:i\notin S(\qq)\}$. It is clear that such minimal prime is
\[P_{S(\qq)}(I):=\Big(I+(x_{i}:i\in S(\qq)): \prod_{i\notin S(\qq)}x_{i}^{\infty}\Big).\]
On the other hand, since $\qq$ is a minimal prime of $I$ and $I\subset I+(x_{i}:i\in S(\qq))\subseteq \qq$, it follows that $\qq$ is a minimal prime of $I+(x_i:i\in S(\qq))$, and since $\qq$ has no indeterminates from $\{x_{i}:i\notin S(\qq)\}$, it follows that $\qq$ is such unique minimal prime. Hence, it must follows that $\qq=P_{S(\qq)}(I)$. This proves that $I$ is coordinate-saturated.
\end{proof}

In the remaining of this section, we fix $I$ to be a coordinate-saturated ideal such that $P_{\emptyset}(I)$ is prime, and thus a minimal prime of $I$. We let $\pp:=P_{\emptyset}(I)$. Note that $\pp$ being prime is equivalent to $I$ having a unique minimal prime containing no indeterminates, in which case it is precisely $\pp$. The goal of this section is to compute $\vn_{\pp}(I)$.

\begin{definition}\normalfont\label{def: transversals}
{\rm A {\it hypergraph} $\mathcal{H}$ is a pair of two sets $(V(\mathcal{H}), E(\mathcal{H}))$, where $V(\mathcal{H})$ is called the vertex set of $\mathcal{H}$ and $E(\mathcal{H})$ is a collection of subsets of $V(\mathcal{H})$, called the edge set of $\mathcal{H}$. A \textit{simple} hypergraph is a hypergraph such that no two elements (called edges) of $E(\mathcal{H})$ contain each other. A simple graph is an example of a simple hypergraph, whose edges are of cardinality two.}
We say that a subset $T\subseteq [n]$ is a {\em transversal} of $\mathcal{H}$ if $T\cap e\neq \emptyset$ for all $e\in \mathcal{H}$. We denote by $\TT(\mathcal{H})$ the set of all transversals of $\mathcal{H}$.
\end{definition}

Note that the term hypergraph in this manuscript always denotes a simple hypergraph.

\begin{definition}\label{ideal M and gamma}
Let
\[
M:=\Big\langle x^{u} : x^{u}\in
\bigcap_{\qq\in\min(I)\setminus\{\pp\}}\qq
\Big\rangle.
\]
Since
\[
(I:\pp)=\bigcap_{\qq\in\min(I)\setminus\{\pp\}}\qq,
\]
$M$ is precisely the monomial ideal generated by the monomials contained in $(I:\pp)$.
Note that the generating set of $M$ is nonempty since, given our assumption above, every $\qq\in \min(I)\setminus \{\pp\}$ contains some indeterminate.
Moreover, let $\gamma(I):=\min\{\deg(x^{u}): x^{u}\in M\}.$
We also define the hypergraph $\Delta$ on the vertex set $[n]$ with set of edges 
\begin{equation}\label{hypergraph delta}E(\Delta):=\{S(\qq):\qq\in \min(I)\setminus \{\pp\}\}.\end{equation}
Again, note that $S(\qq)\neq \emptyset$ for all $\qq\in \min(I)\setminus \{\pp\}$.
One can easily see that 
\begin{equation}\label{M and transversals}M=\langle x^{u}: \text{$\supp(u)$ is a transversal of $\Delta$} \rangle,\end{equation}
and hence
\begin{equation}\label{gamma and transversals}\gamma(I):=\min\{\size{T}: \text{$T$ is a transversal of $\Delta$}\}.\end{equation}
In particular, one can see from this that $M$ is square-free.
\end{definition}

We start by proving that $\gamma(I)$ provides an upper bound for $\vpp$.

\begin{proposition}\label{prop: inequality}
Let $I$ be a coordinate-saturated ideal. Then $\vpp\leq \gamma(I)$.
\end{proposition}

\begin{proof}
Let $x^{u}\in M$ be such that $\deg(x^{u})=\gamma(I)$. Note that, by definition, $x^{u}\in \cap_{\qq\in \min(I)\setminus \{\pp\}}\qq$. Moreover,  $x^{u}\notin \pp$, because $\pp$ is prime and contains no indeterminates. Hence $(I:x^{u})=\pp$, which implies
\[\vpp=\min\{\deg(g): (I:g)=\pp\}\leq \deg(x^{u})=\gamma(I).\]
\end{proof}

We now focus on proving the other inequality, under the assumption that $\KK$ is algebraically closed.

\begin{proposition}\label{prop: I+M rad}
Let $I$ be a coordinate-saturated ideal. If $\KK$ is algebraically closed, then $(I:\pp)=\sqrt{I+M}$.
\end{proposition}

\begin{proof}
It is clear that $I\subset (I:\pp)$. Moreover, $M\subset \cap_{\qq\in \min(I)\setminus \{\pp\}}\qq=(I:\pp)$, and since $(I:\pp)$ is radical, the inclusion $\sqrt{I+M}\subseteq (I:\pp)$ follows.

To show the other inclusion, by Hilbert's Nullstellensatz, it suffices to prove the corresponding inclusion of varieties:
\begin{equation}\label{varieties}
\VV(I+M)\subseteq \closure{\VV(I)\setminus \VV(\pp)}=\bigcup_{\qq\in \min(I)\setminus \{\pp\}}\VV(\qq).
\end{equation}
Let $\vf=(v_{1},\ldots,v_{n})\in \KK^{n}$ be a point in $\VV(I+M)$. If $\vf\notin \VV(\pp)$, then $\vf\in \closure{\VV(I)\setminus \VV(\pp)}$ and there is nothing to prove. We may therefore assume that $\vf\in \VV(\pp)$. We claim that there exists some $\qq\in \min(I)\setminus \{\pp\}$ such that $v_{i}=0$ for all $i\in S(\qq)$. Assume, for contradiction, that this does not hold. Then, for each $\qq\in \min(I)\setminus \{\pp\}$, we can choose $i_{\qq}\in S(\qq)$ such that $v_{i_{\qq}}\neq 0$. If we define
\[T:=\{i_{\qq}: \qq\in \min(I)\setminus \{\pp\}\},\]
then we see that this set contains at least one element from each $S(\qq)$, and hence is a transversal of $\Delta$, the hypergraph from~\Cref{hypergraph delta}. Hence, by~\Cref{M and transversals}, it follows that $\prod_{i\in T}x_{i}\in M$. However, by construction, we have $v_{i}\neq 0$ for all $i\in T$, and hence $\vf\notin \VV(\prod_{i\in T}x_{i})$, which contradicts that $\vf\in \VV(M)$. Hence, there exists some $\qq\in \min(I)\setminus \{\pp\}$ such that $v_{i}=0$ for all $i\in S(\qq)$, or equivalently $\vf\in \VV((x_{i}:i\in S(\qq)))$. We fix such $\qq$. By Lemma~\ref{expression of PSI}, we have
\[
\begin{aligned}
\qq&=\Big((I\cap \KK[x_i:i\notin S(\qq)]): \prod_{i\notin S(\qq)}x_{i}\Big)+(x_{i}:i\in S(\qq)) \\
&\subset \Big(I: \prod_{i\in [n]}x_{i}\Big)+(x_{i}:i\in S(\qq))\\
&=\pp+(x_{i}:i\in S(\qq)).
\end{aligned}
\]

Hence $\qq\subset \pp+(x_{i}:i\in S(\qq))$, and
since $\vf \in \VV(\pp)$ and $\vf \in \VV((x_{i}:i\in S(\qq)))$, it follows that $\vf\in \VV(\qq)$. This shows that $\vf$ belongs to the variety on the right-hand side of~\Cref{varieties}, proving the inclusion in~\Cref{varieties}, as desired.
\end{proof}

We are now ready to prove the main result of this section.

\begin{theorem}\label{thm: main theorem coordinate-saturated ideals}
Let $I$ be a coordinate-saturated ideal, and suppose that $I$ has a unique minimal prime containing no indeterminates. Then this prime is necessarily
\[
\pp=\Big(I:\prod_{i\in [n]}x_i \Big).
\]
Moreover, $\vpp\leq \gamma(I),$ with equality whenever $I+M$ is radical and $\KK$ is algebraically closed. In particular, the equality holds whenever $I$ is binomial.
\end{theorem}

\begin{proof}
The inequality $\vpp \leq \gamma(I)$ follows from Proposition~\ref{prop: inequality}. Now, suppose that $I + M$ is radical. By \textup{\cite[Theorem~10]{grisalde2021induced}},
\[
    \vpp = \min\{\deg(g_{i}) : i \in [r],\ (I:g_{i}) = \pp\},
\]
where $\mathcal{G} = \{\overline{g_{1}}, \ldots, \overline{g_{r}}\}$ is any homogeneous minimal generating set of $(I:\pp)/I$. By Proposition~\ref{prop: I+M is radical}, $(I:\pp) = \sqrt{I+M} = I+M$, so the generators of $M$ form a homogeneous minimal generating set of $(I:\pp)/I$, and therefore
\[
    \vpp = \min\{\deg(u) : u \in M\} = \gamma(I). \qedhere
\]
We now show that if $I$ is binomial, then $I+M$ is radical. We first show that any monomial contained in $I+M$ must belong to $M$. Indeed, if $u\in I+M$ is a monomial, we have
\[u\in I+M=\bigcap_{\qq\in \min(I)}\qq+M\subset \bigcap_{\qq\in \min(I)}\qq+ \bigcap_{\qq\in \min(I)\setminus \{\pp\}}\qq\subset \bigcap_{\qq\in \min(I)\setminus \{\pp\}}\qq. \]
Since $u$ is a monomial in $\cap_{\qq\in \min(I)\setminus \{\pp\}}\qq$, it follows from definition that $u\in M$. Hence, any monomial contained in $I+M$ must belong to $M$. Then, by Proposition~\ref{prop: I+M is radical}, we obtain that $I+M$ is radical.
\end{proof}

\section{v-number of LSS ideals of forest graphs}\label{LSS ideals-forest}
In this section, we focus on computing the $\vn$-number of the LSS ideal $L_G^{\mathbb{K}}(d)$ when $G$ is a forest, $\mathbb{K}$ is an algebraically closed field, and $d \geq 3$. To this end, we apply the framework developed in~\Cref{framework coordinate saturated} to this setting. 

Throughout this section, we fix $G$ to be a forest on the vertex set $[n]$ and $\KK$ to be an algebraically closed field, and we fix $d\geq 3$. In \cite{liwski2025lov}, the author finds the irreducible components of $\OR$ and the primary decomposition of $L_G^{\mathbb{K}}(d)$ in this setting. To present these results, we first need some definitions from \cite{liwski2025lov}.

\begin{definition}\textup{\cite[Definitions~3.1, 3.3 and~3.15]{liwski2025lov}}
For $S\subseteq [n]$, set 
\[U(S):=\{\vf=(v_{1},\ldots,v_{n})\in \OR: v_{i}=0 \text{ if and only if $i\in S$}\},\]
and let $V_{S}=\closure{U(S)}\subset \KK^{nd}$ be its Zariski closure. We also denote
\[U_{G}:=U(\emptyset)=\{\vf=(v_{1},\ldots,v_{n})\in \OR: v_{i}\neq 0 \text{ for all $i\in [n]$}\}.\]
A subset $S\subseteq [n]$ is called $G$-admissible if, for every $i\in S$, $\size{N_{G}(i)\cap ([n]\setminus S)}\geq d.$
\end{definition}

We note that \textup{\cite[Definition~3.15]{liwski2025lov}} includes the additional hypothesis that, for every $i \in S$,
\begin{equation}\label{redundant}
    N_G(i) \cap ([n] \setminus S) \not\subseteq N_G(j) \cap ([n] \setminus S)
    \quad \text{for every } j \in [n] \setminus S.
\end{equation}
This condition is in fact redundant: since $G$ is a forest, it contains no $K_{2,2}$, and so we have $|N_G(i) \cap N_G(j)| \leq 1$ for all $i,j$, while $|N_G(i) \cap ([n] \setminus S)| \geq d$, and together these force~\Cref{redundant} to hold automatically. We denote
\[\mm_{S}:=\langle x_{i,j}: i\in S, j\in [d] \rangle.\]
We use $\bI(V)$ to denote the ideal associated to a variety $V$, consisting of all polynomials that vanish on $V$. Moreover, note that 
\[\bI(V_{S})=\mm_{S}+\bI(U_{G\setminus S}),\]
and we write $\pp_{S}(G)$ for this ideal.

\begin{theorem}\textup{\cite[Theorem~3.18]{liwski2025lov}}\label{thm: deco OR}
The irreducible decomposition of $\OR$ is 
\[\OR=\bigcup_{S}V_{S},\]
where the union ranges over all $G$-admissible subsets of $[n]$.
\end{theorem}

\begin{corollary}\label{cor: primary decomposition LG}
Let $G$ be a forest graph. Then the primary decomposition of $L_G^{\mathbb{K}}(d)$ is $L_G^{\mathbb{K}}(d)=\bigcap_{S}\pp_{S}(G)$,
where the intersection is taken over all $G$-admissible subsets of $[n]$.
\end{corollary}

\begin{proof}
This follows from passing to the ideals the decomposition in Theorem~\ref{thm: deco OR}, noting that $L_G^{\mathbb{K}}(d)$ is radical by \textup{\cite[Theorem~1.5]{cw19}}.
\end{proof}

A more explicit description of the minimal primes $\pp_{S}(G)$ is given in \textup{\cite[\S~5.1]{liwski2025lov}}. In this section, we focus on the following question.

\begin{question}
Compute the localized $\vn$-number $\textup{v}_{\pp_{\emptyset}(G)}(L_G^{\mathbb{K}}(d))$.
\end{question}

We address this question using the framework developed in~\Cref{framework coordinate saturated}, proving that the ideal $L_G^{\mathbb{K}}(d)$ is \emph{coordinate-saturated}.

\begin{lemma}\label{lem: LG is coordinate saturated}
Let $G$ be a forest graph. Then the ideal $L_G^{\mathbb{K}}(d)$ is \emph{coordinate-saturated}. Moreover, $L_G^{\mathbb{K}}(d)$ has a unique minimal prime containing no indeterminates.
\end{lemma}

\begin{proof}
The ideal $L_G^{\mathbb{K}}(d)$ is radical by \textup{\cite[Theorem~1.5]{cw19}}. By the description of the minimal primes of $L_G^{\mathbb{K}}(d)$ from Corollary~\ref{cor: primary decomposition LG}, it follows that the unique minimal prime of $L_G^{\mathbb{K}}(d)$ containing no indeterminates $x_{i,j}$ is $\pp_{\emptyset}(G)$. To show that $L_G^{\mathbb{K}}(d)$ is coordinate-saturated it suffices to show that conditions~\ref{s1} and~\ref{s2} from Proposition~\ref{prop: sufficient conditions for coordinate-saturated} hold. Let $\qq:=\pp_{S}(G)$ be a minimal prime of $L_G^{\mathbb{K}}(d)$. Then, with the notation of Proposition~\ref{prop: sufficient conditions for coordinate-saturated}, we have $S(\qq)=S\times [d]$. With the notation of~\Cref{fe}, for each $e:=\{i,j\}\in E(G)$, we have 
\[
\pi_{S(\qq)}(f_{e}^{(d)})=
\begin{cases}
f_{e}^{(d)} & \text{if $i,j\notin S$;}\\
0 & \text{otherwise}.
\end{cases}
\]
Hence $\pi_{S(\qq)}(I)\subseteq I$ and~\Cref{s1} holds. We now show that~\Cref{s2} holds. Note that
\[L_G^{\mathbb{K}}(d)+(x_{i,j}:(i,j)\in S(\qq))=L_G^{\mathbb{K}}(d)+(x_{i,j}:i\in S,j\in [d])=L_G^{\mathbb{K}}(d)+\mm_{S}=L_{G\setminus S}^{\mathbb{K}}(d)+\mm_{S}.\]
Since each minimal prime of $L_{G\setminus S}^{\mathbb{K}}(d)+\mm_{S}$ arises a sum of $\mm_{S}$ and a minimal prime of $L_{G\setminus S}^{\mathbb{K}}(d)$, and there is a unique minimal prime of $L_{G\setminus S}^{\mathbb{K}}(d)$ containing no indeterminates, it follows that there exists a unique minimal prime of $L_{G\setminus S}^{\mathbb{K}}(d)+\mm_{S}$ containing no indeterminates $\{x_{i,j}:i\notin S\}$. Hence,~\Cref{s2} holds, and $L_G^{\mathbb{K}}(d)$ is coordinate-saturated.
\end{proof}

\begin{notation}
We denote $V_{\geq d}(G):=\{i\in [n]: \deg_{G}(i)\geq d\}\quad \text{and} \quad v_{d}(G):=\size{V_{\geq d}(G)}.$
We define \[I_{\geq d}=\bigcap_{i\in V_{\geq d}(G)}\mm_{\{i\}} \subseteq \mathbb{K}[x_{ij}: i \in [n], j \in [d]].\] Observe that the associated variety is $\VV(I_{\geq d})=\bigcup_{i\in V_{\geq d}(G)}\{(v_{1},\ldots,v_{n})\in (\KK^{d})^{n}: v_{i}=0\}.$
\end{notation}

Our goal in what follows is to prove that $\vp=v_{d}(G)$.

\begin{lemma}\label{lem: LG M and gamma}
Let $G$ be a forest graph. Then the ideal $I_{\geq d}$ and the invariant $v_d(G)$ coincide with the ideal $M$ and the invariant $\gamma(L_G^{\mathbb{K}}(d))$ from Definition~\ref{ideal M and gamma}, respectively.
\end{lemma}

\begin{proof}
By definition, $\{i\}$ is an admissible subset for each $i\in V_{\geq d}(G)$. Hence, with the notation of Definition~\ref{ideal M and gamma}, it follows that
\[S(\pp_{\{i\}}(G))=\{i\}.\]
Moreover, $S(\qq)\subseteq V_{\geq d}(G)$ for all $\qq\in \min(L_G^{\mathbb{K}}(d))$. Hence, the unique minimal transversal of the hypergraph 
\[\{S(\qq):\qq\in \min(L_G^{\mathbb{K}}(d))\setminus \{\pp_{\emptyset}(G)\}\}\]
is $V_{\geq d}(G)$. Both claims follow from this along with~\Cref{M and transversals} and~\Cref{gamma and transversals}.
\end{proof}

We now compute the initial ideal of $L_G^{\mathbb{K}}(d)+I_{\geq d}$ and show that it is radical. 

\begin{proposition}\label{prop: initial ideal and radical}
Let $G$ be a forest graph. Then $\ini(L_G^{\mathbb{K}}(d)+I_{\geq d})=\ini(L_G^{\mathbb{K}}(d))+I_{\geq d}$, for any monomial order $\prec$.
Moreover, $L_G^{\mathbb{K}}(d)+I_{\geq d}$ is radical.
\end{proposition}

\begin{proof}
First note that the polynomial ring $\KK[x_{ij}: i \in [n], j \in [d]]$ has a $\ZZ^{n}$-multigrading given by 
\begin{equation}\label{multi grading}\deg(x_{i,j})=e_{i}\in \ZZ^{n},
\end{equation}
under which the ideal $L_G^{\mathbb{K}}(d)$ is homogeneous. Moreover, with this grading, the ideal $I_{\geq d}$ can be described as 
\[I_{\geq d}=\{u: \text{$u$ is a monomial with $\deg(u)\geq e_{\vg}$}\},\]
where $e_{\vg}$ is the indicator vector of $\vg$. Hence, whenever $u$ and $v$ are monomials with $u\in I_{\geq d}$ and $\deg(u)\leq \deg(v)$, then $v\in I_{\geq d}$. Thus, we can apply Proposition~\ref{prop: sum of initials and upper closed}, which gives that $\ini(L_G^{\mathbb{K}}(d)+I_{\geq d})=\ini(L_G^{\mathbb{K}}(d))+I_{\geq d}$, for any monomial order $\prec$. 

On the other hand, by \textup{\cite[Theorem~6.1]{cw19}}, the ideal $L_G^{\mathbb{K}}(d)$ is Cartwright--Sturmfels, so in particular $\ini(L_G^{\mathbb{K}}(d)$ is square-free, and hence $\ini(L_G^{\mathbb{K}}(d)+I_{\geq d})$ is square-free as well. It follows that $L_G^{\mathbb{K}}(d)+I_{\geq d}$ is radical.
\end{proof}

We now prove the main result of this section.

\begin{theorem}\label{thm:local v-number LSS ideals}
Let $G$ be a forest graph and $d \geq 3$. Then the localized $\vn$-number $\vp$ coincides with $\vd$, that is, the number of vertices of $G$ of degree at least $d$.
\end{theorem}

\begin{proof}
By Lemma~\ref{lem: LG is coordinate saturated}, $L_G^{\mathbb{K}}(d)$ is \emph{coordinate-saturated} and $\pp_{\emptyset}(G)$ is its unique minimal prime containing no indeterminates. Moreover, by Lemma~\ref{lem: LG M and gamma}, we have that the ideal $I_{\geq d}$ and the invariant $v_d(G)$ coincide with the ideal $M$ and the invariant $\gamma(L_G^{\mathbb{K}}(d))$ from Definition~\ref{ideal M and gamma}. Furthermore, by Proposition~\ref{prop: initial ideal and radical} we know that $L_G^{\mathbb{K}}(d)+I_{\geq d}$ is radical. We can then apply Theorem~\ref{thm: main theorem coordinate-saturated ideals}, which gives the desired result.
\end{proof}

Throughout this section, we were assuming that $d\geq 3$ to be able to use the results of \cite{liwski2025lov}. Below we treat the case $d=2$.

\begin{proposition}\label{cor: d=2}
Let $G$ be a forest graph. Then $\vn_{\mathfrak{p}_{\emptyset}(G)}(L_G^{\mathbb{K}}(2))=v_{2}(G)=\size{\iv(G)},$
where $\iv(G)$ is the set of vertices which are neither leaves nor isolated.
\end{proposition}

\begin{proof}
For bipartite graphs, the ideal $L_{G}^{\mathbb{K}}(2)$ is isomorphic to the binomial edge ideal $J_{G}$, and under this isomorphism the minimal prime $\pe$ corresponds to $J_{K_{n}}$. The result then follows from \textup{\cite[Corollary~B]{jaramillo2024connected}}.
\end{proof}
\begin{corollary}\label{cor:bound}
Let $G$ be  a forest graph. Then $\vn(L_G^{\mathbb{K}}(d)) \leq v_d(G)$.    
\end{corollary} 

    

\begin{notation}
Let $G$ be a simple graph and $d \geq 3$. Then we recall the combinatorial invariant $t(G)$ introduced in \cite[Notation 4.3]{nkv26}:
\[
t(G)
=
\max\Bigl\{
|E(H)|
\,\Big|\,
H \text{ is an induced subgraph of } G,\;
H \text{ is a forest, and }
\Delta(H)\le d
\Bigr\},
\]
where $\Delta(G)=\max\{\deg_G(v):v\in V(G)\}.$
\end{notation}

\begin{lemma}\label{lem: inequality vd and tg}
Let $G$ be a forest graph and $d\geq 3$. Then, $v_{d}(G)\leq t(G)$.
\end{lemma}

\begin{proof}
Since both $v_{d}(G)$ and $t(G)$ are additive on connected components, it suffices to treat the case, where $G$ is connected, that is, where $G$ is a tree. We prove the result by induction on $\size{V(G)}$. If $\size{V(G)}=1$, then $v_{d}(G)=t(G)=0$ and the inequality holds. For the inductive step, suppose $v_{d}(G)\leq t(G)$ for any tree on at most $n-1$ vertices and let $G$ be a tree on the vertex set $[n]$. Let 
\[\leaf(G):=\{i\in [n]: \deg_{G}(i)=1\},\]
be the set of leaves of $G$. Suppose there exists some $i\in \leaf(G)$ such that its unique neighbor has degree distinct from $d$. Say $N_{G}(i)=\{j\}$ with $\deg_{G}(j)\neq d$. In this case, we have $j\in V_{\geq d}(G)$ if and only if $j\in V_{\geq d}(G\setminus \{i\})$ and so $v_{d}(G)=v_d(G\setminus \{i\})$. We then have 
\[v_{d}(G)=v_{d}(G\setminus \{i\})\leq t(G\setminus \{i\})\leq t(G),\]
where the first inequality follows from the inductive hypothesis. This gives the required in equality $v_{d}(G)\leq t(G)$.

Hence, we may assume that all vertices in
\[N_{G}(\leaf(G))=\{j\in [n]: \{i,j\}\in E(G) \text{ for some $i\in \leaf(G)$}\}\]
have degree exactly $d$. Consider the induced subgraph $H := G[\leaf(G) \cup N_{G}(\leaf(G))].$
Since every vertex in $\leaf(G)$ has degree $1$ and every vertex in $N_{G}(\leaf(G))$ has degree exactly $d$ in $G$, we have $\Delta(H) \leq d$, and hence $t(G) \geq |E(H)|$ by definition of $t(G)$. Moreover, every vertex in $\leaf(G)$ is a leaf of $H$ and thus belongs to a unique edge, and the edges incident to distinct leaves are distinct, giving $|E(H)| \geq |\leaf(G)|$ and therefore $t(G) \geq |\leaf(G)|$.
On the other hand, since $G$ has $n$ vertices and $n-1$ edges, because it is a tree, and easy counting argument gives $\sum_{v\in [n]}(\deg_{G}(v)-2)=-2.$
Hence,
\[
\begin{aligned}
-2&=\sum_{v\in [n]}(\deg_{G}(v)-2)\\
&=\sum_{v\in \leaf(G)}-1+\sum_{\deg_{G}(v)\geq 2}(\deg_{G}(v)-2)=-\size{\leaf(G)}+\sum_{\deg_{G}(v)\geq 2}(\deg_{G}(v)-2)\\
&\geq -\size{\leaf(G)}+\sum_{\deg_{G}(v)\geq d}(d-2)=-\size{\leaf(G)}+v_{d}(G)(d-2).
\end{aligned}
\]
Consequently,
\begin{equation*}\label{d-2}\size{\leaf(G)}\geq 2+v_{d}(G)(d-2)\geq 2+v_{d}(G),\end{equation*}
where the last inequality holds since $d\geq 3$. Therefore, using the inequality \(t(G)\geq |\leaf(G)|\), we obtain \(v_d(G)\leq t(G)\). This completes the inductive step.
\end{proof}

\begin{theorem}\label{thm:v-number-regularity}
 Let $G$ be a forest graph, and let $d \geq 1$. Then $\vn(L_G^{\mathbb{K}}(d)) \leq \reg(R/L_G^{\mathbb{K}}(d))$.   
\end{theorem}
\begin{proof}
For $d=1$, we get that $L_G^{\mathbb{K}}(d)=I(G)$. Then the required inequality holds using \cite[Theorem 4.5]{ss22}. For $d=2$, using \Cref{cor: d=2} and \cite[Theorem 4.1]{JNR}, we obtain
$$\vn(L_G^{\mathbb{K}}(2)) \leq  \size{\iv(G)} \leq \reg(R/L_G^{\mathbb{K}}(2)).$$ For $d \geq 3$, by \cite[Corollary 4.4]{nkv26}, we have $\reg\!\left(\frac{R}{L_G^{\mathbb{K}}(d)}\right)
\geq t(G).$ Also, by Lemma~\ref{lem: inequality vd and tg}, we have that $v_d(G) \leq t(G)$. Thus, by \Cref{cor:bound}, we obtain $\vn(L_G^{\mathbb{K}}(d)) \leq \reg(R/L_G^{\mathbb{K}}(d))$.
\end{proof}

\section{$\vn$-number of $L_{G}^{\RR}(2)$ over real field $\RR$}\label{LSS ideals real field}

In this section, we study the v-number of $L_{G}^{\RR}(2)$ over real field $\RR$. We achieve this after applying the framework developed in~\Cref{framework coordinate saturated} to this case. We first present the primary decomposition of $L_{G}^{\RR}(2)$ over $\RR$ from \cite{hmsw15}. Throughout, we fix $G$ to be a connected and non-bipartite graph on $[n]$. In the case where $G$ is bipartite, $L_{G}^{\RR}(2)$ is isomorphic to $J_{G}$, see \textup{\cite[Corollary~6.2]{bms24}}.

\begin{definition}\textup{\cite[\S 2]{hmsw15}}
We define $I_{K_{n}}$ as the ideal generated by the binomials
\[
\begin{aligned}
f_{i,j}:&=x_{i}x_{j}+y_{i}y_{j}, \quad 1\leq i<j\leq n,\\
g_{i,j}:&=x_{i}y_{j}-x_{j}y_{i}, \quad 1\leq i <j\leq n,\\
h_{i}:&=x_{i}^{2}+y_{i}^{2}, \quad 1\leq i\leq n.
\end{aligned}
\]
For $1\leq m<n$, we define $I_{K_{m,n-m}}$ as the ideal generated by the binomials
\[
\begin{aligned}
f_{i,j}:&=x_{i}x_{j}+y_{i}y_{j}, \quad 1\leq i\leq m, \quad m+1\leq j\leq n,\\
g_{i,j}:&=x_{i}y_{j}-x_{j}y_{i}, \quad 1\leq i<j\leq m, \quad m+1\leq i<j\leq n.
\end{aligned}
\]
\end{definition}

We now define the building blocks of the primary decomposition of $L_{G}^{\mathbb{R}}(2)$.

\begin{notation}
Let $H$ be an arbitrary connected graph on $[n]$. If $H$ is not bipartite, then we define $\widetilde{H}:=K_{n}$, and if $H$ is bipartite with biconnected components $\{1,\ldots,n\}$ and $\{n+1,\ldots,m\}$, then we set $\widetilde{H}:=I_{K_{n,m-n}}$.
\end{notation}

We denote by $G_{S}$ the induced subgraph of $G$ on $[n]\setminus S$.

\begin{definition}\textup{\cite[\S 4]{hmsw15}}
For $S\subseteq [n]$, set 
\[Q_{S}(G):=\langle \{x_{i},y_{i}\}_{i\in S}, I_{\widetilde{G}_{1}},\ldots, I_{\widetilde{G}_{c(S)}} \rangle, \]
where $G_{1},\ldots,G_{c(S)}$ are the connected components of $G_{S}$.
\end{definition}

\begin{definition}\textup{\cite[\S 5]{hmsw15}}
We say that $i\in [n]$ is a {\em cut point} of $G$ if $G\setminus \{i\}$ has more connected components than $G$. Moreover, we say that $i$ is a {\em bipartite point} if $G\setminus \{i\}$ has more bipartite connected components than $G$. Let $\mathcal{M}(G)$ be the set of all subsets $S\subseteq [n]$ such that each $i\in S$ is either a cut point or a bipartite point of $G_{S\setminus \{i\}}$. In particular, $\emptyset \in \mathcal{M}(G)$.
\end{definition}

\begin{theorem}\textup{\cite[Theorems~1.1 and~5.2]{hmsw15}}\label{minimal primes LG2}
Let $G$ be a simple graph on $[n]$, and $S \subseteq [n]$. Then the ideal $L_{G}^{\mathbb{R}}(2)$ is radical and its minimal prime decompositions is
\[L_{G}^{\mathbb{R}}(2)=\bigcap_{S\in \MG}Q_{S}(G).\]
\end{theorem}

Note that since $G$ is non-bipartite, we have $Q_{\emptyset}(G)=I_{K_{n}}$ is a minimal prime of $L_{G}^{\mathbb{R}}(2)$. Our next goal is to provide a good upper bound for the localized v-number $\mathrm{v}_{I_{K_{n}}}(L_{G}^{\mathbb{R}}(2))$. To address this, we use the framework developed in~\Cref{framework coordinate saturated}, proving that the ideal $L_{G}^{\mathbb{R}}(2)$ is \emph{coordinate-saturated}.

\begin{lemma}\label{lem: LG2 is coordinate saturated}
Let $G$ be a non-bipartite simple graph. Then the ideal $L_{G}^{\mathbb{R}}(2)$ is \emph{coordinate-saturated}. Moreover, $L_{G}^{\mathbb{R}}(2)$ has a unique minimal prime containing no indeterminates, which is precisely $I_{K_{n}}$.
\end{lemma}

\begin{proof}
The second assertion follows from the description of the minimal primes of $L_{G}^{\mathbb{R}}(2)$ from Theorem~\ref{minimal primes LG2}. To show that $L_{G}^{\mathbb{R}}(2)$ is coordinate-saturated it suffices to show that conditions~\ref{s1} and~\ref{s2} from Proposition~\ref{prop: sufficient conditions for coordinate-saturated} hold. Let $\qq:=Q_{S}(G)$ be a minimal prime of $L_{G}^{\mathbb{R}}(2)$. Using the notation from Proposition~\ref{prop: sufficient conditions for coordinate-saturated}, we have $S(\qq)=\{x_{i},y_{i}:i\in S\}$. Applying the exact same argument as in~\Cref{lem: LG is coordinate saturated} we see that the conditions~\ref{s1} and~\ref{s2} hold.
\end{proof}


\begin{definition}\label{dominating-set-gamma-cn}
{\rm
A dominating set of $G$ is a subset $S\subseteq [n]$ such that every vertex in $[n]\setminus S$ is adjacent to some vertex in $S$. We define 
\[\mathcal{D}_{c,n}(G):=\{T\subseteq [n]: \text{ $T$ is dominant and $G[T]$ is connected and non-bipartite}\},
\]
and 
\[\gamma_{c,n}(G):=\min\{\size{T}:T\in \mathcal{D}_{c,n}(G)\}.\]
}
\end{definition}

\begin{lemma}\label{lem: LG2 M and gamma}
Let $G$ be a non-bipartite simple graph. Then the invariant $\gamma_{c,n}(G)$ is greater or equal than the invariant $\gamma(L_{G}^{\mathbb{R}}(2))$ from Definition~\ref{ideal M and gamma}.
\end{lemma}

\begin{proof}
By~\Cref{gamma and transversals}, we have that $\gamma(L_{G}^{\mathbb{R}}(2))$ coincides with the minimum size of a transversal of the hypergraph
\[\Delta:=\{S\subseteq[n]: S\in \MG\}.\]
To show that $\gamma(L_{G}^{\mathbb{R}}(2))\leq \gamma_{c,n}(G)$ it suffices to show that if $T\in \dcng$, then $T$ is a transversal of $\Delta$. Now, suppose that $T\in \dcng$, and assume, for contradiction, that there exists some $S\in \MG\setminus \{0\}$ such that $S\cap T=\emptyset$. Since $T$ is connected and dominant, it follows that $G_{S}$ is connected. Since $S\in \MG$, each $i\in S$ is either a cut point or a bipartite point of $G_{S\setminus \{i\}}$. Since $G_{S}$ is connected, the former can not happen, and since $G_{S}$ is non-bipartite (because $G[T]$ is), it also follows that no $i\in S$ is a bipartite point, a contradiction.
\end{proof}

It is not difficult to show, using combinatorial arguments, that the invariants $\gamma_{c,n}(G)$ and $\gamma(L_{G}^{\mathbb{R}}(2))$ are in fact equal. However, for the purposes of the next theorem, we only need the inequality $\gamma(L_{G}^{\mathbb{R}}(2))\leq \gamma_{c,n}(G)$, which is given by the previous lemma.

\begin{theorem}\label{LSS-over-real}
Let $G$ be a non-bipartite simple graph. Then $\vlg \leq \gamma_{c,n}(G).$
\end{theorem}

\begin{proof}
By Lemma~\ref{lem: LG2 is coordinate saturated}, $L_{G}^{\mathbb{R}}(2)$ is \emph{coordinate-saturated} and $I_{K_n}$ is its unique minimal prime containing no indeterminates. Then, applying Theorem~\ref{thm: main theorem coordinate-saturated ideals}, we obtain that $\vlg\leq \gamma(L_{G}^{\mathbb{R}}(2))$. The result then follows from Lemma~\ref{lem: LG2 M and gamma}, which gives $\gamma(L_{G}^{\mathbb{R}}(2))\leq \gamma_{c,n}(G)$.
\end{proof}

The below corollary is immediate.
\begin{corollary}
Let $G$ be a non-bipartite simple graph. Then $\vn(L_{G}^{\mathbb{R}}(2))\leq \gamma_{c,n}(G)$.
\end{corollary}

\begin{question}
For which graphs $G$ does it hold that $\vlg=\gamma_{c,n}(G)$?
\end{question}

\section{v-number of parity binomial edge ideals}\label{parity binomial edge ideals}

In this section, we study the v-number of the parity binomial edge ideal $\II_{G}$. We do this by following the framework developed in~\Cref{framework coordinate saturated}. Note that parity binomial edge ideals are not coordinate-saturated, however, most of the methods presented in~\Cref{framework coordinate saturated} also apply to this setting, as we will see in this section.
Throughout this section, we fix $G$ to be a connected and non-bipartite graph with vertex set $[n]$ and we fix $\KK$ an algebraically closed field with $\textup{char}(\KK)\neq 2$. In the case where $G$ is bipartite, the ideal $\II_{G}$ is isomorphic to the binomial edge ideal of $G$, see \textup{\cite[Corollary~6.2]{bms24}} and the problem reduces to the v-number of binomial edge ideals of bipartite graphs, which is well studied, see \cite{ass24, dey2024v, jaramillo2024connected, liwski2025v}. First, we present some of the results and definitions from \cite{kst16}. Let $G$ be a simple graph on $[n]$. Then we define its saturation $\JJ_{G}$ as follows:
\[
\JJ_{G}:=\II_{G}:\Big(\prod_{i\in [n]}x_{i}y_{i}\Big)^{\infty}.
\]

\begin{proposition}\cite[Proposition~2.7]{kst16}\label{prop: generators of JG} Let $G$ be a simple graph. Then
\[
\begin{aligned}
\JJ_G
={}&
\bigl\langle
x_i x_j-y_i y_j
:\text{ there exists an odd $(i,j)$-walk in }G
\bigr\rangle \\
&+
\bigl\langle
x_i y_j-x_j y_i
:\text{ there exists an even $(i,j)$-walk in }G
\bigr\rangle .
\end{aligned}
\]
\end{proposition}

We also define the ideals
\[\pp^{+}(G):=\langle x_{i}+y_{i}: i \in [n] \rangle
\quad \text{and} \quad
\pp^{-}(G):=\langle x_{i}-y_{i}: i \in [n] \rangle.
\]

\begin{notation}\label{CGS}For a subset $S\subseteq [n]$, $\mm_{S}:=\langle x_{s},y_{s}:s\in S\rangle$. We also let $c_{0}(G)$ and $c_{1}(G)$ denote the number of bipartite and non-bipartite connected components of $G$, respectively. Denote by $\mathcal{C}_{G_{S}}(s)$ the set of those connected components of $G_{S}$ which are joined when adding $s$.
\end{notation}

\begin{proposition}\cite[Proposition~4.2]{kst16}\label{prop: minimal primes of JG}
Let $G$ be a graph consisting of bipartite connected components $B_{1},\ldots,B_{c_{0}(G)}$ and non-bipartite connected components $N_{1},\ldots,N_{c_{1}(G)}$. Then the ideal $\JJ_{G}$ is radical and its minimal primes are the $2^{c_{1}(G)}$ ideals
\[\sum_{i=1}^{c_{0}(G)}\JJ_{B_{i}}+\sum_{i=1}^{c_{1}(G)}\pp^{\sigma_{i}}(N_{i}),\]
where $\sigma$ ranges over $\{+,-\}^{c_{1}(G)}$.
\end{proposition}

\begin{definition}\cite[Definitions~4.5 and~4.11]{kst16}\normalfont \label{def: sign split}
\ Let $\mathfrak{s}(G)=2c_{0}(G)+c_{1}(G)$. A set $S\subset [n]$ is a {\em disconnector} of $G$ if $\mathfrak{s}(G_{S})>\mathfrak{s}(G_{S\setminus \{s\}})$ for all $s\in S$.  A minimal prime $\pp$ of $\JJ_{G_{S}}$ is called {\em sign-split} if for all $s\in S$ such that $\mathcal{C}_{G_{S}}(s)$ contain no bipartite graphs, the prime summands of $\pp$ corresponding to connected components in $\mathcal{C}_{G_{S}}(s)$ are not all equal to $\pp^{+}$ or all equal to $\pp^{-}$. 
\end{definition}

The following theorem from \cite{kst16} presents the primary decomposition of $\II_{G}$.

\begin{theorem}\label{thm:minimal primes}\cite[Theorem~4.15]{kst16}
The minimal primes of $\II_{G}$ are the ideals $\mm_{S}+\pp$, where $S\subset [n]$ is a disconnector of $G$ and $\pp$ is a sign-split minimal prime of $\JJ_{G_{S}}$.
\end{theorem}

Given that we work under the assumption that 
$\text{char}(\mathbb{K}) \neq 2$, we can obtain radicality.

\begin{theorem}\cite[Theorem~5.5]{kst16}
The ideal $\II_{G}$ is radical.
\end{theorem}

Since the empty set is a disconnector and $G$ is non-bipartite, it follows that the ideals 
\[\pp^{+}(G):=\langle x_{i}+y_{i}: i \in [n] \rangle
\quad \text{and} \quad
\pp^{-}(G):=\langle x_{i}-y_{i}: i \in [n] \rangle
\]

are minimal primes of $\II_{G}$. In this section, we focus on the following question.

\begin{question}\label{question}
Compute the localized $\vn$-number $\vpi$.
\end{question}

Note that $\vpi=\textup{v}_{\pp^{-}(G)}(\II_{G})$
since the involution $y_{i}\mapsto -y_{i}$ maps $\pp^{+}(G)$ to $\pp^{-}(G)$. We start by introducing some notation and definitions.

\begin{definition}\normalfont\label{def: q-}
For each minimal prime $\qq\in\min(\II_G)$, we associate the subsets
\[
\qq(+):=\{\,i\in[n] : x_i+y_i\in\qq\,\},\qquad
\qq(-):=\{\,i\in[n] : x_i-y_i\in\qq\,\}.
\]
We define the hypergraph $\Delta_{G}$ on $[n]$ as follows:
\[
\Delta_{G}:=\{\qq(-): \qq\in \min(\II_{G})\setminus \{\pp^{+}(G)\}\}.
\]
We also introduce the combinatorial invariant $\nu_{G}$ as follows:
\[\nu_{G}:=\min \{\size{T}: \text{ $T$ is a transversal of $\Delta_{G}$}\}.\]
Note that $\nu_G$ is well defined, since every edge of the form
$\qq(-)$ in $\Delta_G$ is nonempty. Indeed, every minimal
prime $\qq\in\min(\II_G)$ distinct from $\pp^{+}(G)$ contains 
a polynomial of the form $x_i-y_i$ 
for some $i\in[n]$; see \Cref{thm:minimal primes}.
\end{definition}

Our goal in what follows is to prove that $\vpi=\nu_{G}$. We start by proving that $\nu_{G}$ provides an upper bound for $\vpi$. The following proposition is an adaptation of Proposition~\ref{prop: inequality}, but does not follow from it since $\II_{G}$ is not coordinate-saturated.

\begin{proposition}\label{prop: less or equal}
Let $G$ be a non-bipartite simple graph. Then $\vpi \leq \nu_{G}$.
\end{proposition}

\begin{proof}
Let $T\subset [n]$ be a transversal of $\Delta_{G}$ of minimum size, such that $\nu_{G}=\size{T}$. Let
\[f_{T}:=\prod_{i\in T}(x_{i}-y_{i}).\]
We will show that $(\II_{G}: f_{T})=\pp^{+}(G)$. To prove this we must show that $\pp^{+}(G)f_{T}\subset \II_{G}$ and $f_{T}\notin \pp^{+}(G)$. We first show that $\pp^{+}(G)f_{T}\subset \II_{G}$. Let $\qq$ be any minimal prime of $\II_{G}$ distinct from $\+$. Since $T$ is a transversal of $\Delta_{G}$, it follows that there exists some $i\in T$ such that $x_{i}-y_{i}\in \qq$. Since $x_{i}-y_{i}$ is a factor of $f_{T}$, it follows that $f_{T}\in \qq$. Moreover, since $\+ f_{T}\subset \+$, it follows that
\[\+ f_{T}\subset \bigcap_{\qq\in \min(\II_{G})}\qq=\II_{G}.\]
On the other hand, note that none of the factors of $f_{T}$ belong to $\+$, and since $\+$ is prime, it follows that $f_{T}\notin \+$. This proves that $(\II_G:f_T)=\pp^{+}(G)$.
Consequently,
\[\vpi=\min\{\deg(f): (\II_G:f)=\pp^{+}(G)\} \leq \deg(f_T) = |T| = \nu_G.\]
\end{proof}

We now focus on proving the other inequality. For this, we introduce a relevant monomial ideal. First, we introduce some notation.

\begin{notation}
For subsets $C,D\subseteq [n]$, we define $\displaystyle g_{C,D}:=\prod_{i\in C}x_{i}\prod_{j\in D}y_{j}\in \mathbb{K}[\bf{x},\bf{y}].$
\end{notation}

Recall the notion of transversals and of $\TT(H)$ from Definition~\ref{def: transversals}.

\begin{definition}
We define the monomial ideal associated to $\Delta_G$ by
\[
I_{\Delta_G}
:=
\bigl\langle
g_{C,D}
:
C\cup D\in \TT(\Delta_G),\;
C\cap D=\varnothing
\bigr\rangle
\subseteq
\mathbb{K}[\mathbf{x},\mathbf{y}].
\]
\end{definition}

The following result is an adaptation of Proposition~\ref{prop: I+M rad}, but does not follow from it.

\begin{proposition}\label{prop:saturation and radical}
Let $G$ be a non-bipartite simple graph. Then $(\II_{G}:\pp^{+}(G))\subseteq \sqrt{\II_{G}+I_{\Delta_{G}}}.$
\end{proposition}

\begin{proof}
By Hilbert's Nullstellensatz, it suffices to prove the corresponding
inclusion of varieties:
\begin{equation*}
\VV(\II_G+I_{\Delta_G})
\subseteq
\overline{\VV(\II_G)\setminus \VV(\pp^{+}(G))}
=
\bigcup_{\qq\in\min(\II_G)\setminus\{\pp^{+}(G)\}}
\VV(\qq).
\end{equation*}
Let $\mathbf{v}=(p_{1},\ldots,p_{n},q_{1},\ldots,q_{n})\in \mathbb{K}^{2n}$ be a point in $\mathbb{V}(\II_{G}+I_{\Delta_{G}})$. If $\mathbf{v}\notin\VV(\pp^{+}(G))$, then $\mathbf{v}\in
\VV(\II_G)\setminus\VV(\pp^{+}(G))$,
and there is nothing to prove. We may therefore assume that
$\mathbf{v}\in\VV(\pp^{+}(G))$. Suppose, for contradiction, that
\[
\mathbf{v}\notin
\VV(\qq)
\qquad
\text{for all }
\qq\in\min(\II_G)\setminus\{\pp^{+}(G)\}.
\] 
By \Cref{thm:minimal primes}, every such prime is of the form
$\qq=\mm_S+\pp$, where $\pp$ is a minimal prime of $\JJ_{G_S}$.
Combining Propositions~\ref{prop: generators of JG}
and~\ref{prop: minimal primes of JG}, we see that $\qq$ admits a
generating set consisting of polynomials of the form
\begin{equation}\label{six polynomials}
x_i,\qquad
y_i,\qquad
x_i-y_i,\qquad
x_i+y_i,\qquad
x_ix_j-y_iy_j,\qquad
x_iy_j-x_jy_i.
\end{equation}
Since $\mathbf{v}\notin\VV(\qq)$, there exists a generator
$f_{\qq,\mathbf v}\in\qq$ that does not vanish at $\mathbf v$.
Observe that
\[
x_ix_j-y_iy_j=(x_i+y_i)x_j-y_i(x_j+y_j)
\]
and
\[
x_iy_j-x_jy_i=(x_i+y_i)y_j-y_i(x_j+y_j),
\]
hence both binomials belong to $\pp^{+}(G)$. The same is true for
$x_i+y_i$. Since $\mathbf v\in\VV(\pp^{+}(G))$, none of these
polynomials can serve as $f_{\qq,\mathbf v}$.

 Therefore, $f_{\qq,\mathbf v}$ must be of the form
$x_i$, $y_i$, or $x_i-y_i$ for some $i\in[n]$.
Denote such an index by $i_{\qq,\mathbf v}$.
By ~\Cref{def: q-}, we have
$i_{\qq,\mathbf v}\in\qq(-)$.

Consider now the polynomial
\[
f_{\mathbf v}
:=
\prod_{\qq\in\min(\II_G)\setminus\{\pp^{+}(G)\}}
f_{\qq,\mathbf v}.
\]
Since none of the factors vanish at $\mathbf v$, we have
$f_{\mathbf v}(\mathbf v)\neq 0$. On the other hand, every factor of $f_{\mathbf v}$ is of the form
\[
x_{i_{\qq,\mathbf v}},
\qquad
y_{i_{\qq,\mathbf v}},
\qquad\text{or}\qquad
x_{i_{\qq,\mathbf v}}-y_{i_{\qq,\mathbf v}}.
\]
Consequently, each monomial appearing in the expansion of
$f_{\mathbf v}$ is equal to $g_{C,D}$ for some partition
$\{C,D\}$ of the set
\[
T_{\mathbf v}
=
\{\,i_{\qq,\mathbf v}
:
\qq\in\min(\II_G)\setminus\{\pp^{+}(G)\}\,\}.
\]
Since $i_{\qq,\mathbf v}\in\qq(-)$ for every
$\qq\in\min(\II_G)\setminus\{\pp^{+}(G)\}$, the set $T_{\mathbf v}$
is a transversal of $\Delta_G$. Hence every monomial of
$f_{\mathbf v}$ belongs to $I_{\Delta_G}$, and therefore
$f_{\mathbf v}\in I_{\Delta_G}$. Because $\mathbf v\in\VV(I_{\Delta_G})$, it follows that $f_{\mathbf v}(\mathbf v)=0$, contradicting the fact that
$f_{\mathbf v}(\mathbf v)\neq 0$. Therefore,
$$\mathbf{v} \in \VV(\qq) \qquad \text{for some } \qq\in\min(\II_G)\setminus \{\pp^{+}(G)\}.$$
Hence, we obtain that $(\II_{G}:\pp^{+}(G))\subseteq \sqrt{\II_{G}+I_{\Delta_{G}}}.$
\end{proof}

We now show that $\II_{G}+I_{\Delta_{G}}$ is radical and compute its initial ideal.

\begin{proposition}\label{prop: radical and ini IG}
Let $G$ be a non-bipartite simple graph. Then the ideal $\II_{G}+I_{\Delta_{G}}$ is radical. Moreover, $\ini(\II_{G}+I_{\Delta_{G}})=\ini(\II_{G})+I_{\Delta_{G}}$ for any monomial order $\prec$.
\end{proposition}

\begin{proof}
We first show that any monomial contained in $\II_{G}+I_{\Delta_{G}}$ belongs to $I_{\Delta_{G}}$. Consider the monomial map $\phi$ defined by
\[\phi:\mathbb{K}[x_{1},\ldots,x_{n},y_{1},\ldots,y_{n}]\rightarrow \mathbb{K}[z_{1},\ldots,z_{n}],\qquad x_{i},y_{i}\mapsto z_{i}.
\]
Note that $\phi(x_{i}x_{j}-y_{i}y_{j})=0$ for any $i,j\in [n]$, and hence $I\subseteq \ker(\phi)$. On the other hand, by the definition of $I_{\Delta_{G}}$, it follows that if $u$ and $v$ are monomials with $u\in I_{\Delta_G}$ and $\phi(u)= \phi(v)$, then $v\in I_{\Delta_G}$ as well. We can apply Proposition~\ref{prop: sufficiento condition for monomials in I+M}, which gives that any monomial contained in $\II_{G}+I_{\Delta_{G}}$ belongs to $I_{\Delta_{G}}$. Consequently, since $I_{\Delta_{G}}$ is binomial and radical, and $I_{\Delta_{G}}$ is monomial and radical, we can apply Propositions~\ref{prop: I+M is radical} and~\ref{prop: ini(I+M)}, which show that $\II_{G}+I_{\Delta_{G}}$ is radical and that $\ini(\II_{G}+I_{\Delta_{G}})=\ini(\II_{G})+I_{\Delta_{G}}$ for any monomial order $\prec$.
\end{proof}

We are now in position to prove the other inequality.

\begin{proposition}\label{prop: greater or equal}
Let $G$ be a non-bipartite simple graph. Then $\vpi\geq \nu_{G}$.
\end{proposition}

\begin{proof}
Let $f$ be a homogeneous polynomial satisfying $(\II_{G}:f)=\+$. Then we must show that
\begin{equation*}\label{deg f}
\deg(f)\geq \nu_{G}.
\end{equation*}
Writing $f$ in normal form modulo $\II_{G}$, with respect to a monomial order $\prec$, we obtain a homogeneous decomposition $$f=g+h, \qquad g\in \II_{G},$$
where either $h=0$ or $\ini(h)\notin \ini(\II_{G})$. Since $f\notin \II_{G}$,
necessarily $h\neq 0$, and hence $\ini(h)\notin \ini(\II_{G}).$
Moreover, since $g\in \II_{G}$, it follows that $(\II_{G}:h)=(\II_{G}:f)=\+$ and since the decomposition is homogeneous, we also have $\deg(h)=\deg(f)$. By \Cref{prop:saturation and radical}, we have
\[h\in (\II_{G}:\+)\subseteq \sqrt{\II_{G}+I_{\Delta_{G}}}=\II_{G}+I_{\Delta_{G}}, \]
where the last equality follows from \Cref{prop: radical and ini IG}. We then have
\[\ini(h)\in \ini(\II_{G}+I_{\Delta_{G}})=\ini(\II_{G})+I_{\Delta_{G}},\]
where the last equality follows from Proposition~\ref{prop: radical and ini IG}. Using that $\ini(h)\notin \ini(\II_{G})$, it follows that $\ini(h)\in I_{\Delta_{G}}$. Hence, $\ini(h)$ is divisible by some generator of $I_{\Delta_{G}}$. Hence, there exist subsets $C,D\subset [n]$ with $C\cup D\in \TT(\Delta_{G})$ and $C\cap D=\emptyset$ such that $g_{C,D}\mid \ini(h).$
It follows that 
\[\deg(\ini(h))\geq \deg(g_{C,D})=\size{C\cup D}\geq \min \{\size{T}:T\in \TT(\Delta_{G})\}=\nu_{G}.\]
Note that $\deg(f)=\deg(h)=\deg(\ini(h)).$ Hence, we get that $\deg(f) \geq \nu_{G}$, as required.
\end{proof}

We conclude this section with its main result.

\begin{theorem}\label{thm: vpi equals nuG}
Let $G$ be a non-bipartite simple graph. The localized $\vn$-number $\vpi$ of $\II_{G}$ at the minimal prime $\+$ coincides with the combinatorial invariant $\nu_{G}$. In other words, $\vpi=\nu_{G}.$
\end{theorem}

\begin{proof}
This follows directly from Propositions~\ref{prop: less or equal} and~\ref{prop: greater or equal}, which give each of the inequalities.
\end{proof}

\section{Explicit formula for $\vpi$}\label{explicit formula for parity binomial edge ideals}

In Theorem~\ref{thm: vpi equals nuG}, we proved that $\vpi$ coincides with the combinatorial invariant $\nu_{G}$. However, $\nu_{G}$ is not completely explicit since it is defined as the minimum size of a transversal of $\Delta_{G}$. In this section, we will provide an explicit combinatorial description of this invariant. As in the previous section, we assume throughout that $G$ is connected and non-bipartite. Recall that for $S \subseteq [n]$, $G[S]$ denote the induced subgraph of $G$ on $S$. Note that $G[S]=G_{[n]\setminus S}$.

\begin{definition}\normalfont
A {\em dominant set} of $G$ is a subset $S\subseteq [n]$ such that every vertex in $[n]\setminus S$ is adjacent to some vertex in $S$. We say that a graph is {\em totally non-bipartite} if all its connected components are non-bipartite. We define
\[\mathcal{D}_{n}(G):=\{S\subseteq[n]: \text{$S$ is dominant and $G[S]$ is totally non-bipartite}\},\]
and $\gamma_{n}(G):=\min \{\size{S}: S\in \mathcal{D}_{n}(G)\}.$
\end{definition}

Our goal in what follows is to prove that $\nu_{G}=\gamma_{n}(G)$. We will do this by showing that $\TT(\Delta_{G})=\dng$. To that end, we need the following lemma. Recall the notion of $\mathcal{C}_{G_{S}}(s)$ from Notation~\ref{CGS}. For $s\notin S$, we  denote by $\mathcal{C}_{G[S]}(s)$ the set of those connected components of $G[S]$ which are joined when adding $s$.

\begin{lemma}\label{lem: totally non-bipartite}
Let $S\subseteq [n]$ be such that $G[S]$ is not totally-non bipartite. Then there exists a nonempty disconnector
$S'\subseteq [n]\setminus S$ such that, for every $s\in S'$,
the collection $\mathcal{C}_{G_{S'}}(s)$ contains at least one
bipartite graph.
\end{lemma}

\begin{proof}
Let $B_{1},\ldots,B_{c_{0}(G[S])}$ and $N_{1},\ldots,N_{c_{1}(G[S])}$ be the bipartite and non-bipartite connected components of $G[S]$ respectively.
We say that a vertex $j\notin S$ is {\em nice with respect to $S$} if at least one of the following holds:
\begin{enumerate}[label=(\roman*)]
\item\label{cond1} $\mathcal{C}_{G[S]}(j)=\{B_{i}\}$ for some $1\leq i\leq c_{0}(G[S])$ and $B_{i}\cup \{j\}$ is bipartite.
\item\label{cond2} $\mathcal{C}_{G[S]}(j)\subseteq \{N_{1},\ldots,N_{c_{1}(G[S])}\}$
\end{enumerate}
{\bf Claim~1.} If $G[S]$ is not totally non-bipartite and $j\notin S$ is nice with respect to $S$, then $G[S\cup \{j\}]$ is also not totally non-bipartite. 

\smallskip
By definition, $j$ satisfies either~\ref{cond1} or~\ref{cond2}. A straightforward case analysis shows that none of these conditions can destroy a bipartite connected component of $G[S]$. Consequently, $G[S\cup\{j\}]$ still contains a bipartite connected component and, in particular, is not totally non-bipartite. 

\medskip

Let $j_{1}\in [n]\setminus S$ be any vertex that is nice with respect to $S$. Since $G[S]$ is not totally non-bipartite and by Claim~1, we have that $G[S\cup \{j_{1}\}]$ is also not totally non-bipartite. Iterating this procedure, we can find a sequence of vertices $j_{1},\ldots,j_{k}$ such that $j_{r}$ is nice with respect to $S\cup \{j_{1},\ldots,j_{r-1}\}$ for all $r\in [k]$ and such that each $j\notin S\cup \{j_{1},\ldots,j_{k}\}$ is not nice with respect to $S$. We claim that $S':=[n]\setminus (S\cup \{j_{1},\ldots,j_{k}\})$ satisfies the desired properties.

By construction, we have $S'\subseteq [n]\setminus S$. Furthermore,
Claim~1 implies that $G_{S'}$ is not totally non-bipartite. In
particular, $S'$ must be nonempty, since $G=G_{\emptyset}$ is connected
and non-bipartite, and hence totally non-bipartite. We now show that $S'$ is a disconnector. For that, we must show that
\begin{equation}\label{inequality with s}2c_{0}(G_{S'})+c_{1}(G_{S'})>2c_{0}(G_{S'\setminus \{s\}})+c_{1}(G_{S'\setminus\{s\}}) \quad \text{for all $s\in S'$}.\end{equation}
By construction, we know that $s$ is not nice with respect to $[n]\setminus S'$, so both~\ref{cond1} and~\ref{cond2} do not hold. Hence, we must have that 
\begin{equation*}\label{nonempty}
\mathcal{C}_{G_{S'}}(s)\cap \{B_{1},\ldots,B_{c_{0}(G_{S'})}\}\neq \emptyset
\end{equation*}
and if $\mathcal{C}_{G[S]}(s)=\{B_{i}\}$ for some $1\leq i\leq c_{0}(G[S])$ then $B_{i}\cup \{s\}$ is non-bipartite. If $\size{\mathcal{C}_{G_{S'}}(s)}\geq 2$, then $s$ joins at least two of the connected components of $G_{S'}$ when being added, and hence at least one of $c_{0}$ or $c_{1}$ decreases, which proves~\Cref{inequality with s}. If $\size{\mathcal{C}_{G_{S'}}(s)}\leq 1$, then $\mathcal{C}_{G[S]}(s)=\{B_{i}\}$ for some $1\leq i\leq c_{0}(G[S])$ and $B_{i}\cup \{s\}$ is non-bipartite. In this case, we have $c_{0}(G_{S'\setminus \{s\}})<c_{0}(G_{S'})$, and thus~\Cref{inequality with s} also holds. It follows that $S'$ is a disconnector. 

To conclude, it follows from $\mathcal{C}_{G_{S'}}(s)\cap \{B_{1},\ldots,B_{c_{0}(G_{S'})}\}\neq \emptyset$ that, for every $s\in S'$,
the collection $\mathcal{C}_{G_{S'}}(s)$ contains a bipartite graph.
Therefore, $S'$ satisfies all the required properties, and the proof is
complete.
\end{proof}

\begin{corollary}\label{cor: not a transversal}
Let $S\subseteq [n]$ be such that $G[S]$ is not totally-non bipartite. Then, $S\notin \TT(\Delta_{G})$.
\end{corollary}

\begin{proof}
By Lemma~\ref{lem: totally non-bipartite}, there exists a nonempty disconnector
$S'\subseteq [n]\setminus S$ such that, for every $s\in S'$,
the collection $\mathcal{C}_{G_{S'}}(s)$ contains at least one
bipartite graph.
Let $B_{1},\ldots,B_{c_{0}(G_{S'})}$ and $N_{1},\ldots,N_{c_{1}(G_{S'})}$ be the bipartite and non-bipartite connected components of $G_{S'}$ respectively. By Proposition~\ref{prop: minimal primes of JG}, the ideal
\[\pp:=\sum_{i}^{c_{0}(G_{S'})}\JJ_{B_{i}}+\sum_{i=1}^{c_{1}(G_{S'})}\pp^{+}(N_{i})\]
is a minimal prime of $\JJ_{G_{S'}}$. Moreover, since for every $s\in S'$,
the collection $\mathcal{C}_{G_{S'}}(s)$ contains at least one
bipartite graph, it follows that $\pp$ is sign-split, according to Definition~\ref{def: sign split}. Hence, by Theorem~\ref{thm:minimal primes}, it follows that $\qq:=\mm_{S'}+\pp$ is a minimal prime of $\II_{G}$. Moreover, by the description of $\pp$, it follows that $\qq(-)=S'$. By definition of $\Delta_{G}$, we have $S'=\qq(-)\in \Delta_{G}$, and since $S'\subseteq [n]\setminus S$, the claim follows.
\end{proof}

\begin{proposition}\label{prop: combinatorial equality}
Let $G$ be a non-bipartite simple graph. Then $\TT(\Delta_{G})=\dng.$
\end{proposition}

\begin{proof}
We first show that $\TT(\Delta_{G}) \subseteq \dng.$ Let $T\in \TT(\Delta_{G})$. We must show that $T$ is a dominating set and that $G[T]$ is totally non-bipartite. By Corollary~\ref{cor: not a transversal}, it follows that $G[T]$ is totally non-bipartite. Assume, for contradiction, that $T$ is not a dominating set. Then, there exists some vertex $j\notin N_{G}(T)$. Hence, $\{j\}$ is a connected component of $G[T\cup \{j\}]$, which implies that $T\cup \{j\}$ is not totally non-bipartite. Again by Corollary~\ref{cor: not a transversal}, this implies that $T\cup \{j\}\notin \TT(\Delta_{G})$, which implies that $T\notin \TT(\Delta_{G})$, a contradiction. This proves the inclusion.

We now show the inclusion $\TT(\Delta_{G}) \supseteq \dng.$ Let $T\in \dng$. Assume, for contradiction, that $T\notin \TT(\Delta_{G})$. Then, there exists $\qq\in \min(\II_{G})\setminus \+$ such that $T\cap \qq(-)=\emptyset$. By Theorem~\ref{thm:minimal primes}, $\qq$ must be of the form $\mm_{S}+\pp$, where $S\subseteq [n]$ is a disconnector and $\pp$ is a sign-split minimal prime of $\JJ_{G_{S}}$. Since $S\subseteq \qq(-)$, it follows that $T\subseteq [n]\setminus S$. 

\smallskip
{\bf Claim.} Every connected component of $G_{S}$ intersects $T$ and is non-bipartite.

\smallskip
Since $T\subseteq [n]\setminus S$, and $T$ is dominant, it follows that every connected component of $G_{S}$ intersects $T$. Let $G_{1}$ be any of these components. By the already proven, it must contain some vertex $t\in T$, and hence the whole connected component of $G[T]$ containing $t$ is also a subset of the vertices of $G_{1}$. Since all the connected components of $G[T]$ are non-bipartite, it follows that $G_{1}$ is non-bipartite as well.

\medskip
Now take any $s\in S$. By the claim, all connected components of $G_{S}$ are non-bipartite, in particular those included in $\mathcal{C}_{G_{S}}(s)$. Thus, since $\pp$ is sign-split, the prime summands of $\pp$ corresponding to the connected components in $\mathcal{C}_{G_{S}}(s)$ are not all equal to $\pp^{+}$. Hence, there exists at least one connected component of $G_{S}$, say $H$, such that the prime summand of $\pp$ corresponding to $H$ is not $\pp^{+}$. Therefore, it follows that $V(H)\subseteq \qq(-)$, where $V(H)$ denotes the vertex set of $H$. By the claim we know that $V(H)\cap T\neq \emptyset$ and hence $T\cap \qq(-)\neq \emptyset$, which contradicts our assumption that $T\cap \qq(-)=\emptyset$. This completes the proof. 
\end{proof}

We are now in position to provide an explicit formula for $\vpi$.

\begin{theorem}\label{thm: vpi equals yng}
Let $G$ be a non-bipartite simple graph. Then 
\[\vpi=\gamma_{n}(G)=\min \{\size{S}: \text{$S$ is a dominating set and $G[S]$ is totally non-bipartite}\}.\]
\end{theorem}

\begin{proof}
By \Cref{prop: combinatorial equality}, we have that $\TT(\Delta_{G})=\dng$ and hence
\[\nu_{G}=\min\{\size{T}:T\in \TT(\Delta_{G})\}=\min\{\size{T}:T\in \dng\}=\gamma_{n}(G).\]
The result then follows from \Cref{thm: vpi equals nuG}, which gives $\vpi=\nu_{G}$.
\end{proof}

The following corollary is immediate.
\begin{corollary}\label{cor: vpi inequals yng}
Let $G$ be a non-bipartite simple graph. Then $\vn(\II_G) \leq \gamma_{n}(G).$
\end{corollary}

\section{v-number and regularity of parity binomial edge ideals}\label{v-number versus regularity for parity binomial edge ideals}

In this section, we study the following question.

\begin{question}\label{question 2}
For which graphs $G$ does it hold that $\vn(\II_{G})\leq \textup{reg}(R/\II_{G})?$
\end{question}

We will focus on this question for a specific family of graphs, which we now introduce. We first recall the notion of clique-sum of graphs.

\begin{definition}\normalfont
Let $G_1$ and $G_2$ be two subgraphs of a graph $G$. Suppose that $G_1 \cap G_2 = K_m,$ where $K_m$ is the complete graph on $m$ vertices, and $G_1 \neq K_m$ and $G_2 \neq K_m$. Then $G$ is called the \emph{$m$-clique sum} of $G_1$ and $G_2$ along the complete graph $K_m$, and is denoted by $G = G_1 \cup_{K_m} G_2.$ If $m=1$, then the clique sum of $G_1$ and $G_2$ along a vertex is denoted by $G = G_1 \cup G_2.$ 
\end{definition}

\begin{definition}\normalfont\label{def: decorated tree}
A \emph{decorated tree} is a graph $G$ obtained from a tree $T$ by choosing a subset
$S\subseteq V(T)$ of pairwise non-adjacent vertices and, for each $v\in S$,
attaching an odd cycle $C_v$ to $T$ via a $1$-clique sum, that is, $G = T \cup C_v,$ where a vertex of $C_v$ is identified with the vertex $v$ of $T$.



We assume that the cycles $\{C_v : v\in S\}$ are pairwise disjoint outside of
their attachment vertices or attachment edges in $T$. The tree $T$ is called the \emph{underlying tree} of $G$, and the vertices in
$S$ are called the \emph{decorated vertices}.
\end{definition}
Observe that the classes of graphs appearing in \cite[Theorems 4.4 and 4.5]{jks20} and \cite[Theorem 4.1]{nk24} are all decorated trees. We now aim to show that $\vn(\II_G)\leq \reg(R/\II_G)$
whenever \(G\) is a decorated tree. To that end, we introduce some notation.

\begin{notation}\normalfont
Let $G$ be a decorated tree and let $T$ be the underlying tree. Fix a vertex $r\in V(T)$ and view $T$ as a rooted tree. This induces a partial order
\begin{equation}\label{weak order}u\leq_{r}v \quad \text{if $u$ lies in the unique path from $v$ to $r$}.\end{equation}
We say that a decorated vertex $v\in V(T)$ is {\em maximal} if there does not exist another decorated vertex $v'\in V(T)$ such that $v<_{r}v'$. For a maximal decorated vertex $v$, we define
\begin{equation}\label{C tilde}\widetilde{C}_{v}:=C_{v}\cup \{u: u\geq_{r}v\}\subset V(G).\end{equation}
Moreover, for every vertex $w \in V(T)\setminus\{r\}$, we define its
\emph{parent} as the unique neighbour of $w$ lying on the path from $w$
to the root $r$.
\end{notation}





We denote by $\iv(G)$ the set of internal vertices of $G$, i.e. the set
of vertices that are not leaves. For an edge $e\in E(G)$, we denote by $G\setminus e$ the graph on the vertex set $V(G)$ with edge set $E(G)\setminus \{e\}$.

\begin{lemma}\label{lem: unique decorated vertex}
Let $G$ be a decorated tree with a unique decorated vertex. Then
\[\reg(R/\II_{G})\geq \size{\iv(G)}.\]
\end{lemma}

\begin{proof}
Let $v$ be the unique decorated vertex of $G$ and let $e=\{r,s\}$ be any edge of the odd cycle $C_{v}$ not containing $v$. Consider the short exact sequence
\begin{align}\label{exact-sequence}
0 \longrightarrow
\frac{R}{\II_{G\setminus e}:g_e}(-2)
\xrightarrow{\cdot g_e}
\frac{R}{\II_{G\setminus e}}
\longrightarrow
\frac{R}{\II_G}
\longrightarrow 0,
\end{align}
where $g_{e}=g_{\{r,s\}}=x_{r}x_{s}-y_{r}y_{s}$. Since $G$ is non-bipartite and $G\setminus e$ is bipartite, using \textup{\cite[Lemma~3.4]{ar4}}, it follows that 
\[\II_{G\setminus e}:g_{e}=\II_{G\setminus e}+\langle x_{i}y_{j}-x_{j}y_{i}: i,j\in N_{G\setminus e}(r) \text{ or } i,j\in N_{G\setminus e}(s) \rangle.
\]
Since $r$ and $s$ have degree one in $G\setminus e$, it follows that 
$\II_{G\setminus e}:g_{e}=\II_{G\setminus e}.$ Thus, we get that
$$\reg\!\left(\frac{R}{\II_{G\setminus e}:g_{e}}\right)+2\neq \reg\!\left(\frac{R}{\II_{G\setminus e}}\right).$$
Then by applying \textup{\cite[Lemma~3.1]{zbMATH05708656}} on exact sequence \eqref{exact-sequence}, we have 
$$\reg(R/\II_G)
=
\max\left\{
\reg\!\left(\frac{R}{\II_{G\setminus e}:g_e}\right)+1,\,
\reg\!\left(\frac{R}{\II_{G\setminus e}}\right)
\right\}.$$
 Therefore, we obtain that
\begin{equation*}\label{G and Ge}
\reg(R/\II_G)=\reg(R/\II_{G\setminus e})+1.
\end{equation*}

As $G\setminus e$ is bipartite, \cite[Corollary~6.2]{bms24} implies that $\II_{G\setminus e}$ is isomorphic to the binomial edge ideal $J_{G\setminus e}$ of $G\setminus e$. Furthermore, since $G\setminus e$ is a tree, it follows from \cite[Theorem~4.1]{JNR} that
\[
\reg(R/J_{G\setminus e})
\geq
|\iv(G\setminus e)|+1
=
|\iv(G)|-1.
\]
Therefore, we have
\[
\reg(R/\II_G)
=
\reg(R/\II_{G\setminus e})+1
=
\reg(R/J_{G\setminus e})+1
\geq
|\iv(G)|.
\]
This completes the proof.
\end{proof}

We are now in position to show that $\vn(\II_{G})\leq \reg(R/ \II_{G})$.

\begin{theorem}\label{thm:decorated-tree}
Let $G$ be a decorated tree. Then $\vn(\II_{G})\leq \reg(R/\II_{G}).$
\end{theorem}

\begin{proof}
Let $T$ be the underlying tree of $G$ with $S\subseteq V(T)$ the set of decorated vertices. If $S=\emptyset$, then $G$ is a tree, and in particular bipartite. Hence, \cite[Corollary~6.2]{bms24} implies that $\II_{G}$ is isomorphic to the binomial edge ideal $J_{G}$ of $G$. Moreover, since $G$ is a tree, by \Cref{cor: d=2} and \Cref{thm:v-number-regularity}, we have $\vn(J_{G})\leq\reg(R/J_{G})$, from which we obtain $\vn(\II_G) \leq \reg(R/\II_G)$.

Assume now that $S\neq \emptyset$. Then we will show that 
\begin{equation*}\label{yng}
\gamma_{n}(G)\leq \reg(R/\II_{G})
\end{equation*}
as proving this implies  $\vn(\II_{G}) \leq \gamma_{n}(G) \leq \reg(R/\II_G)$ holds by~\Cref{cor: vpi inequals yng}. We proceed by induction $\size{S}\geq 1$ to show the inequality $\gamma_{n}(G)\leq \reg(R/\II_{G})$. For the base case, let $\size{S}=1$. In this case, by Lemma~\ref{lem: unique decorated vertex}, we have $\reg(R/\II_{G})\geq \size{\iv(G)}.$
On the other hand, $\iv(G)$ is connected, dominant and non-bipartite, so we have $\iv(G)\geq \gamma_{n}(G)$. Thus by \Cref{cor: vpi inequals yng}, we have
$$\vn(\II_{G}) \leq \gamma_{n}(G)\leq \size{\iv(G)} \leq \reg(R/\II_{G}). $$
For the inductive step, assume that $\size{S}\geq 2$ and that the inequality $\gamma_{n}(G)\leq \reg(R/\II_{G})$ holds for any decorated tree with at most $\size{S}-1$ decorated vertices. Fix a vertex $r\in V(T)$ and view $T$ as a rooted tree. Let $v\in S$ be a maximal decorated vertex, with respect to the weak order in~\Cref{weak order}, and let $w$ be its unique parent. With the notation from~\Cref{C tilde}, and using that $v$ is maximal, it follows that $\widetilde{C}_{v}$ is a decorated tree with $v$ as its unique decorated vertex. Hence, applying Lemma~\ref{lem: unique decorated vertex} to it, we obtain
\begin{equation*}\label{reg Cv}
\reg(R/\II_{G[\widetilde{C}_{v}]})\geq \lvert\iv(G[\widetilde{C}_{v}])\rvert.
\end{equation*}
On the other hand, by definition, no two adjacent vertices of $G$ are decorated, so $w$ is not decorated. It then follows that $N_{G}(\widetilde{C}_{v})=\{w\}$ and that $G_{\widetilde{C}_{v}\cup \{w\}}$ is a decorated tree. Since $G_{\widetilde{C}_{v}\cup \{w\}}$ has one less decorated vertex than $G$, it follows from inductive hypothesis that 
\begin{equation*}\label{induc hypothesis}
\reg(R/\II_{G_{\widetilde{C}_{v}\cup \{w\}}})\geq \gamma_{n}(G_{\widetilde{C}_{v}\cup \{w\}}).
\end{equation*}
Note that the induced subgraph $G_{w}=G[V(G)\setminus \{w\}]$ has connected components $G_{\widetilde{C}_{v}\cup \{w\}}$ and $G[\widetilde{C}_{v}]$. Since $G_{w}$ is induced, it follows from \textup{\cite[Theorem~3.7]{zbMATH07352276}} that 
\[\reg(R/\II_{G})\geq \reg(R/\II_{G_{w}})=\reg(R/\II_{G[\widetilde{C}_{v}]})+\reg(R/\II_{G_{\widetilde{C}_{v}\cup \{w\}}}).\]
Thus, we have $$\reg(R/\II_{G})\geq \lvert\iv(G[\widetilde{C}_{v}])\rvert+ \gamma_{n}(G_{\widetilde{C}_{v}\cup \{w\}}).$$ Therefore, to prove that $\gamma_{n}(G)\leq \reg(R/\II_G)$, it is enough to show that
$$\lvert\iv(G[\widetilde{C}_{v}])\rvert+ \gamma_{n}(G_{\widetilde{C}_{v}\cup \{w\}})\geq \gamma_{n}(G).$$
Let $T\subset V(G)\setminus (\widetilde{C}_{v}\cup \{w\})$ be any dominant subset of vertices of $G_{\widetilde{C}_{v}\cup \{w\}}$ such that $G[T]$ is totally non-bipartite and such that $\gamma_{n}(G_{\widetilde{C}_{v}\cup \{w\}})=\size{T}$. Since the subsets of vertices $\widetilde{C}_{v}$ and $V(G)\setminus (\widetilde{C}_{v}\cup \{w\})$ form the connected components of $G_{w}$, it follows that $T\cup \iv(G[\widetilde{C}_{v}])$ is also dominant and totally non-bipartite on $G_{w}$. Moreover, since $w$ is adjacent to $v$, it follows that it is also dominant in $G$. Hence, 
\[\gamma_{n}(G)\leq \lvert T\cup \iv(G[\widetilde{C}_{v}])\rvert=\size{T}+\lvert\iv(G[\widetilde{C}_{v}])\rvert=\gamma_{n}(G_{\widetilde{C}_{v}\cup \{w\}})+\lvert\iv(G[\widetilde{C}_{v}])\rvert.\]
Therefore, we obtain that $\gamma_{n}(G) \leq \reg(R/\II_G)$.
\end{proof}

\begin{proposition}\label{prop:path-triangle}
Let \(G\) be a simple graph on \([n]\) obtained by attaching a path of length at least \(1\) to each vertex of a triangle. Then $\vn(\II_G) \leq \reg(R/\II_G)$.
\end{proposition}
\begin{proof}
Let $\{x_{1},y_{1},z_{1}\}$ be the vertices of the triangle, and let
\[P_{1}:=x_{1},\ldots,x_{i}, \quad P_{2}:=y_{1},\ldots,y_{j} \quad \text{and} \quad P_{3}:=z_{1},\ldots,z_{k}\]
be the paths attached to the triangle $\{x_{1},y_{1},z_{1}\}$.
Note that $\{x_i,y_j,z_k\}$ is the set of leaves of $G$. Since $[n]\setminus \{x_i,y_j,z_k\}$ is non-bipartite, connected and dominant, it follows that $\gamma_{n}(G)\leq n-3$. Thus, ~\Cref{cor: vpi inequals yng} implies that $\vn(\II_{G})\leq \vpi=\gamma_{n}(G)\leq n-3$.

Consider the induced subgraph $G\setminus \{x_1\}$. This graph has connected components $G_{1}$ and $G_{2}$ given by the vertices
\[V_{1}:=\{x_{2},\ldots,x_{i}\}\quad \text{and} \quad V_{2}:=\{y_{1},\ldots,y_{j}\}\cup \{z_{1},\ldots,z_{k}\}.\]
Since $G_{1}$ and $G_{2}$ are path graphs of lengths $i-1$ and $j+k$, respectively, it follows from \textup{\cite[Theorem~3.2]{zbMATH07352276}} that $\reg(R/\II_{G_{1}})=i-2$ and $\reg(R/\II_{G_{2}})=j+k-1$. Since $G_{1}$ and $G_{2}$ are the connected components of $G\setminus \{x_{1}\}$, it follows that $\reg(R/\II_{G\setminus \{x_{1}\}})=i+j+k-3=n-3$. Hence, since $G\setminus \{x_1\}$ is induced, it follows from \textup{\cite[Theorem~3.7]{zbMATH07352276}} that 
\[\reg(R/\II_{G})\geq \reg(R/\II_{G\setminus \{x_{1}\}})= n-3.\]
Hence, $\vn(\II_{G})\leq n-3\leq \reg(R/\II_{G})$, as desired.
\end{proof}

\begin{proposition}\label{prop:chord-even-cycle}
Let $G$ be a non-bipartite graph on $[n]$ obtained by adding a chord in
an even cycle $C_n$. Then $\vn(\II_G) \leq \reg(R/\II_G)$.
\end{proposition}
\begin{proof}
Let $\{1,\ldots,n\}$ and $\{\{i,i+1\}:i\in [n]\}$ be the vertices and edges of $C_{n}$, respectively, and let $\{1,j\}$ be the added chord with $1\leq i<j\leq n$. Since $n$ is even, the intervals $[i+1,j-1]$ or $[j+1,i-1]$ both have an even number of vertices or both have an odd number of vertices. If both intervals are even then $G$ is bipartite, and \cite[Corollary~6.2]{bms24} implies that $\II_{G}$ is isomorphic to the binomial edge ideal $J_{G}$ of $G$. Applying \textup{\cite[Corollary~3.7]{jaramillo2024connected}}, we obtain that $\vn(J_{G})\leq \gamma_{c}(G)$, where $\gamma_{c}(G)$ is the minimal size of a connected and dominant subset of vertices of $G$, which in this case is clearly at most $n-1$, since all but one vertex give a connected and dominant set. Hence $\vn(\II_{G})\leq n-1$ in this case.

If both intervals $[i+1,j-1]$ and $[j+1,i-1]$ are odd, then the induced subgraph on the vertices $[n]\setminus \{i\}$ is connected, dominant and non-bipartite given that it contains the odd cycle $j+1,j+2,\ldots,i-1,i,j$. It follows that $\gamma_{n}(G)\leq n-1$. Thus, ~\Cref{cor: vpi inequals yng} implies that $\vn(\II_{G})\leq \gamma_{n}(G)\leq n-1$. On the other hand, by \cite[Theorem 5.13]{sz23}, we have $\reg(R/\II_G)=n-1$. This proves the desired inequality.
\end{proof}

\begin{proposition}\label{prop:chord-odd-cycle}
Let $G$ be a graph obtained from an odd cycle $C_n$ by adding a chord. Then, $\vn(\II_G) \leq \reg(R/\II_G)$.
\end{proposition}
\begin{proof}
Let $\{1,\ldots,n\}$ and $\{\{i,i+1\}:i\in [n]\}$ be the vertices and edges of $C_{n}$, respectively. Let $\{i,j\}$ be the added chord with $1\leq i<j\leq n$. Since $n$ is odd, at least one of the intervals $[i+1,j-1]$ or $[j+1,i-1]$ has an odd number of vertices. Assume, without loss of generality that $[j+1,i-1]$ has an odd number of vertices and that $[i+1,j-1]$ has an even number of vertices, which has size at least two then. Let $\{k,k+1\}$ be two internal vertices in the interval $[i+1,j-1]$. Note the the induced subgraph on the vertices $[n]\setminus \{k,k+1\}$ is connected, dominant, and non-bipartite given that it contains the odd cycle $j+1,j+2,\ldots,i-1,i,j$. It follows that $\gamma_{n}(G)\leq n-2$. Thus, ~\Cref{cor: vpi inequals yng} implies that $\vn(\II_{G})\leq \gamma_{n}(G)\leq n-2$. On the other hand, by \cite[Theorem 4.6]{jks20} we have $n-2 \leq \reg(R/\II_G)$, which gives $\vn(\II_G) \leq \reg(R/\II_G)$.
\end{proof}

\section{v-number of generalized binomial edge ideals}\label{generalized binomial edge ideals}
In this section, we discuss the $\vn$-number of generalized binomial edge ideals. Recall definition of generalized binomial edge ideals from ~\Cref{def: generalized binomial edge ideals}. 
For each subset $T\subseteq [n]$, we define the ideal
\begin{equation}\label{PT}P_{T}(K_{m},G):=\langle x_{i,j}: i\in [m], j\in T \rangle+J_{K_{m},\widetilde{G}_1}+\cdots +J_{K_{m},\widetilde{G}_{c(T)}},\end{equation}
where $G_{1},\ldots,G_{c(T)}$ are the connected components of the subgraph induced by $G$ on $[n]\setminus T$ and $\widetilde{G}_i$ denotes the complete graph on the vertices of $G_{i}$ for each $1\leq i\leq c(T)$.

For simplicity, we will denote the localized v-number of $J_{K_{m},G}$ with respect to $P_{T}(K_{m},G)$ by $\vn_{T}(J_{K_{m},G})$. Our goal in this section, is to show that the following sequence of integer numbers is weakly-decreasing:
\begin{equation}\label{sequence}\vn(J_{G}), \vn(J_{K_{2},G}),\vn(J_{K_3, G}),\ldots \ldots\end{equation}
\begin{theorem}\textup{\cite[Corollary~4 and Theorem~7]{Rauh13}}\label{thm: decomposition generalized ideals}
The ideal $J_{K_{m},G}$ is radical and its minimal prime decomposition is $J_{K_{m},G}=\cap_{T}P_{T}(K_{m},G)$, where $T$ ranges over all cut sets of $G$.
\end{theorem}

To that end, we need the following proposition.

\begin{proposition}\label{prop: I and J}
Let $R_{1}=\KK[\mathbf{x}]$ and $R_{2}=\KK[\mathbf{x},\mathbf{y}]$ be two polynomial rings, and let $\pi:R_{2}\rightarrow R_{1}$ be a surjective $\KK$-algebra homomorphism such that $\pi(x)=x$ for all $x\in \mathbf{x}$. Let $I\subset R_{1}$ and $J\subset R_{2}$ be two radical ideals such that their minimal prime decompositions are:
\[I=\cap_{i=1}^{r}Q_{i}, \quad J=\cap_{i=1}^{r}Q_{i}',\]
such that, for each $i\in [r]$, we have $Q_{i}\subseteq Q_{i}'|_{R_{1}}$ and $\pi(Q_{i}')\subseteq Q_{i}$. Then, for any $i\in [r]$, we have $\vn_{Q_{i}'}(J)\leq \vn_{Q_{i}}(I)$. In particular, $\vn(J)\leq \vn(I)$.
\end{proposition}

\begin{proof}
Let $f\in R_{1}$ be a homogeneous polynomial such that $I:f=Q_{i}$ and $\deg(f)=\vn_{Q_{i}}(I)$. We claim that $J: f=Q_{i}'$. To prove this we must show that $f\in \cap_{j\neq i}Q_{j}'$ and $f\notin Q_{i}'$. On the other hand, since $I:f=Q_{i}$, it follows that $f\in \cap_{j\neq i}Q_{j}$ and $f\notin Q_{i}$. Hence, since $Q_{j}\subseteq Q_{j}'|_{R_{1}}$, it follows that $f\in \cap_{j\neq i}Q_{j}'$. Now, suppose for contradiction, that $f\in Q_{i}'$. In this case, applying the map $\pi$, we obtain 
\[f=\pi(f)\in \pi(Q_{i}')\subseteq Q_{i},\]
where for the first equality we use that $\pi|_{R_{1}}=\textup{Id}$. However, this is a contradiction, since by assumption $f\notin Q_{i}$. We then have 
\[\vn_{Q_{i'}}(J)=\min\{\deg(g): I:g=Q_{i}'\}\leq \deg(f)=\vn_{Q_i}(I).\]
It follows that 
\[\vn(J)=\min\{\vn_{Q_{i'}}(J):i\in [r]\}\leq \min\{\vn_{Q_{i}}(I):i\in [r]\}=\vn(I).\]
\end{proof}

We can now prove that the sequence in Equation~\eqref{sequence} is weakly-decreasing.

\begin{theorem}\label{thm: monotonicity}
Let $m_{1}\leq m_{2}$ and let $T$ be a cut set of $G$. Then, we have $\vn_{T}(J_{K_{m_2},G})\leq \vn_{T}(J_{K_{m_1},G})$. In particular, we have
\[\vn_{T}(J_{G})\geq \vn_{T}(J_{K_{3},G})\geq \vn_{T}(J_{K_4, G})\geq \cdots \cdots,\]
which implies 
\[\vn(J_{G})\geq \vn(J_{K_{3},G})\geq \vn(J_{K_4, G})\geq \cdots \cdots\]
\end{theorem}

\begin{proof}
Let $R_{1}:=\KK[x_{i,j}:i\in [m_1],j\in [n]]$ and $R_{2}:=\KK[x_{i,j}:i\in [m_2],j\in [n]]$. It suffices to show that the ideals $J_{K_{m_1},G}$ and $J_{K_{m_2},G}$ satisfy the hypothesis of Proposition~\ref{prop: I and J}. By Theorem~\ref{thm: decomposition generalized ideals} both ideals are radical and have minimal prime decompositions given by:
\[J_{K_{m_1},G}=\bigcap_{T}P_{T}(K_{m_1},G) \quad \text{and} \quad J_{K_{m_2},G}=\bigcap_{T}P_{T}(K_{m_2},G),\]
where $T$ ranges over all cut sets of $G$. From the description of the minimal primes in Equation~\eqref{PT}, it is clear that $P_T(K_{m_1},G)=P_T(K_{m_2},G)|_{R_{1}}$, and moreover, if $\pi: R_{2}\rightarrow R_{1}$ is the surjective $\KK$-algebra homomorphism given by $\pi(x_{i,j})=x_{i,j}$ if $i\leq m_{1}$ and $\pi(x_{i,j})=0$ if $m_1+1\leq i\leq m_2$, one can easily see that $\pi(P_T(K_{m_2},G))=P_T(K_{m_1},G)$. Hence, the hypothesis of Proposition~\ref{prop: I and J} are satisfies, and the claim follows.
\end{proof}

We conclude this section proving that the generalized binomial edge ideals $J_{K_{m},G}$ are coordinate-saturated ideals, and we compute the localized v-number $\vn_{\emptyset}(J_{K_{m},G})$ using the framework developed in~\Cref{framework coordinate saturated}.

\begin{lemma}\label{lem: Jkm is coordinate saturated}
The ideal $\jkm$ is \emph{coordinate-saturated}, and $P_{\emptyset}(\jkm)$ is its unique minimal prime containing no indeterminates.
\end{lemma}
\begin{proof}
The ideal $\jkm$ is radical by \textup{\cite[Corollary~4]{Rauh13}}, and the uniqueness of $P_{\emptyset}(\jkm)$ follows immediately from Theorem~\ref{thm: decomposition generalized ideals}. It remains to verify conditions~\ref{s1} and~\ref{s2} of Proposition~\ref{prop: sufficient conditions for coordinate-saturated}. Let $\qq := P_{T}(\jkm)$ be a minimal prime, so that $S(\qq) = [m] \times T$. Since
\[
\pi_{S(\qq)}(x_{i,k}x_{j,l} - x_{i,l}x_{j,k}) =
\begin{cases}
x_{i,k}x_{j,l} - x_{i,l}x_{j,k} & \text{if } k, l \notin T, \\
0 & \text{otherwise},
\end{cases}
\]
we have $\pi_{S(\qq)}(I) \subseteq I$, so~\Cref{s1} holds. For~\Cref{s2}, observe that
\[
\jkm + (x_{i,j} : (i,j) \in S(\qq)) = J_{K_{m}, G \setminus T} + (x_{i,j} : i \in [m],\, j \in T).
\]
Since every minimal prime of the right-hand side is a sum of $(x_{i,j} : i\in[m], j\in T)$ with a minimal prime of $J_{K_m, G\setminus T}$, and the latter has a unique minimal prime containing no indeterminates, there is a unique minimal prime of $\jkm + (x_{i,j}:(i,j)\in S(\qq))$ containing none of the indeterminates $\{x_{i,j} : j \notin T\}$. Hence~\Cref{s2} holds.
\end{proof}

\begin{definition}
We define
\[\mathcal{D}_{c}(G):=\{S\subseteq [n]: \text{$S$ is dominant and $G[S]$ is connected}\},
\]
and
\[\gamma_{c}(G):=\min\{\size{S}:S\in \mathcal{D}_{c}(G)\}.\]
\end{definition}

The following result was already established in \textup{\cite[Theorem~3.3]{sz26}}. However, here we present a different proof using the framework developed in~\Cref{framework coordinate saturated}.

\begin{theorem}\label{thm:generalized v-number}
Let $G$ be a simple graph. Then $\vn_{\emptyset}(J_{K_{m},G})=\gamma_c(G).$
\end{theorem}

\begin{proof}
By Lemma~\ref{lem: LG2 is coordinate saturated}, $\jkm$ is coordinate-saturated and $P_{\emptyset}(\jkm)$ is its unique minimal prime containing no indeterminates. Then Theorem~\ref{thm: main theorem coordinate-saturated ideals} gives $\vn_{\emptyset}(J_{K_{m},G}) = \gamma(\jkm)$, which by~\Cref{gamma and transversals} equals the minimum size of a transversal of
\[
    \Delta := \{S(\qq) : \qq \in \min(I) \setminus \{\pp\}\} = \{T \subseteq [n] : \text{$T$ is a nonempty cut set of $G$}\},
\]
where the last equality follows from Theorem~\ref{thm: decomposition generalized ideals}. The result now follows from \textup{\cite[Lemma~3.7]{liwski2025v}}, which shows that the set of transversals of $\Delta$ is precisely $\mathcal{D}_{c}(G)$.
\end{proof}

\section{v-number and regularity of powers of binomial edge ideals of closed graphs}\label{closed graphs}
 In this section, we compare $\vn$-number and regularity of binomial edge ideals of closed graphs. In \cite[Theorem 6.5]{sz26}, authors compute $\vn$-number of powers of binomial edge ideals of Cohen-Macaulay closed graphs.  The below proposition is related to any closed graphs.
\begin{proposition}\label{proppowers}
 Let $G$ be a closed graph. Then $\vn(J_G^k) \leq \reg(R/J_G^k)$ for all $k \geq 1$.     
\end{proposition}
\begin{proof}
By \cite[Corollary 2.14]{vrt22}, we have $J_G$ satisfies the strong persistence property. Then by \cite[Remark 4.7]{kns25}, $\vn(J_G^k) \leq \vn(J_G)+(k-1)2$. Note that by \cite[Corollary 3.17]{ass24} $\vn(J_G) \leq \vn(\inn(J_G))$. Thus, we get that
 $\vn(J_G^k) \leq \vn(\inn(J_G))+(k-1)2$. Since $\inn(J_G)$ is the edge ideal of bipartite graph $H$, denoted by $I(H)$. By using \cite[Theorem 4.5]{ss22}, $\vn(\inn(G))=\vn(I(H)) \leq \text{im}(H)$. Then  $\vn(J_G^k) \leq \text{im}(H)+(k-1)2$. Thus by \cite[Theorem 4.5]{bht15},  $$\vn(J_G^k) \leq \reg(I(H)^k)-1=\reg(R/I(H)^k)=\reg(R/\inn(J_G)^k).$$ Also, by \cite[Equation (3)]{vrt22}, we have $\inn(J_G^k)=\inn(J_G)^k$ for all $k \geq 1$ and by \cite[Theorem 3.1]{vrt22},  $\reg(R/J_G^k)= \reg(R/\inn(J_G^k))$ for all $k \geq 1$. Therefore, we have $\vn(J_G^k) \leq \reg(R/J_G^k)$ for all $k \geq 1$.
\end{proof}

The below example conveys that $\vn(J_G^k)$ may not be equal to $\reg(R/J_G^k)$.
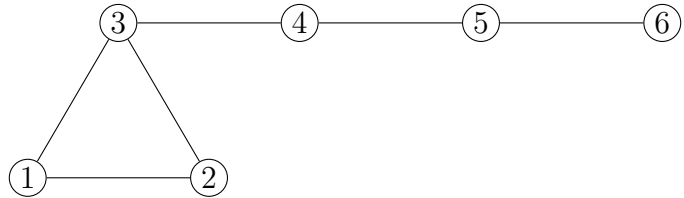
\begin{figure}[h]
\centering
\begin{tikzpicture}[scale=1.2,
every node/.style={circle,draw,inner sep=1.5pt}]

\node (1) at (0,0) {$1$};
\node (2) at (2,0) {$2$};
\node (3) at (1,1.7) {$3$};
\node (4) at (3,1.7) {$4$};
\node (5) at (5,1.7) {$5$};
\node (6) at (7,1.7) {$6$};

\draw (1)--(2);
\draw (2)--(3);
\draw (1)--(3);

\draw (3)--(4);
\draw (4)--(5);
\draw (5)--(6);

\end{tikzpicture}
\caption{The Cohen-Macaulay closed graph $G$.}
\label{figure1}
\end{figure}
\begin{example}
Let $G$ be a Cohen-Macaulay closed graph as in \Cref{figure1}. By \cite[Corollary 5.11]{sz26}, $\vn(J_G)=2$ and by \cite[Theorem 6.5]{sz26},  $\vn(J_G^k)=2k$ and by \cite[Theorem 3.1]{vrt22}, $\reg(R/J_G^k)=2(k+1)$.
\end{example}

\bigskip
\noindent\textbf{\bf Acknowledgement.}  
The first author would like to thank Prof. Volkmar Welker for valuable discussions on Lovász--Saks--Schrijver (LSS) ideals during his research visit to Philipps-Universität Marburg, Germany, and for the warm hospitality extended to him throughout his one-month stay. He also gratefully acknowledges his postdoctoral mentor, Prof. A. V. Jayanthan, for many valuable discussions on binomial edge ideals. The second author was supported by the PhD Fellowship 1126125N and partially funded by the FWO grants G0F5921N (Odysseus), G023721N, and the KU Leuven grant iBOF/23/064.

\bibliographystyle{abbrv}
\bibliography{Reference}

@article{herzog2010binomial,
  title={Binomial edge ideals and conditional independence statements},
  author={Herzog, J{\"u}rgen and Hibi, Takayuki and Hreinsd{\'o}ttir, Freyja and Kahle, Thomas and Rauh, Johannes},
  journal={Advances in Applied Mathematics},
  volume={45},
  number={3},
  pages={317--333},
  year={2010},
  publisher={Elsevier}
}

@article{ohtani2011graphs,
  title={Graphs and ideals generated by some 2-minors},
  author={Ohtani, Masahiro},
  journal={Communications in Algebra{\textregistered}},
  volume={39},
  number={3},
  pages={905--917},
  year={2011},
  publisher={Taylor \& Francis}
}

@article{grisalde2021induced,
  title={Induced matchings and the v-number of graded ideals},
  author={Grisalde, Gonzalo and Reyes, Enrique and Villarreal, Rafael H},
  journal={Mathematics},
  volume={9},
  number={22},
  pages={2860},
  year={2021},
  publisher={MDPI}
}

@article{jaramillo2024connected,
  title={Connected domination in graphs and v-numbers of binomial edge ideals},
  author={Jaramillo-Velez, Delio and Seccia, Lisa},
  journal={Collectanea Mathematica},
  volume={75},
  number={3},
  pages={771--793},
  year={2024},
  publisher={Springer}
}

@article{dey2024v,
  title={On the v-number of binomial edge ideals of some classes of graphs},
  author={Dey, Deblina and Jayanthan, AV and Saha, Kamalesh},
  journal={International Journal of Algebra and Computation},
  pages={1--25},
  year={2024},
  publisher={World Scientific}
}

@article{kst16,
  title={Parity binomial edge ideals},
  author={Kahle, Thomas and Sarmiento, Camilo and Windisch, Tobias},
  journal={Journal of Algebraic Combinatorics},
  volume={44},
  number={1},
  pages={99--117},
  year={2016},
  publisher={Springer}
}

@article {eisenbud1996binomial,
    AUTHOR = {Eisenbud, David and Sturmfels, Bernd},
     TITLE = {Binomial ideals},
   JOURNAL = {Duke Math. J.},
  FJOURNAL = {Duke Mathematical Journal},
    VOLUME = {84},
      YEAR = {1996},
    NUMBER = {1},
     PAGES = {1--45},
      ISSN = {0012-7094,1547-7398},
   MRCLASS = {13P10 (13A30 14M25)},
  MRNUMBER = {1394747},
MRREVIEWER = {P.\ Schenzel},
       DOI = {10.1215/S0012-7094-96-08401-X},
       URL = {https://doi.org/10.1215/S0012-7094-96-08401-X},
}

@article {kns25,
    AUTHOR = {Kumar, Manohar and Nanduri, Ramakrishna and Saha, Kamalesh},
     TITLE = {The slope of the {${\rm v}$}-function and the {W}aldschmidt
              constant},
   JOURNAL = {J. Pure Appl. Algebra},
  FJOURNAL = {Journal of Pure and Applied Algebra},
    VOLUME = {229},
      YEAR = {2025},
    NUMBER = {2},
     PAGES = {Paper No. 107881, 13},
      ISSN = {0022-4049,1873-1376},
   MRCLASS = {13F20 (05E40 13A02 13F55)},
  MRNUMBER = {4853484},
MRREVIEWER = {Th\'ai\ Th\`anh\ Nguy\cftil en},
       DOI = {10.1016/j.jpaa.2025.107881},
       URL = {https://doi.org/10.1016/j.jpaa.2025.107881},
}

@article {vrt22,
    AUTHOR = {Ene, Viviana and Rinaldo, Giancarlo and Terai, Naoki},
     TITLE = {Powers of binomial edge ideals with quadratic {G}r\"obner
              bases},
   JOURNAL = {Nagoya Math. J.},
  FJOURNAL = {Nagoya Mathematical Journal},
    VOLUME = {246},
      YEAR = {2022},
     PAGES = {233--255},
      ISSN = {0027-7630,2152-6842},
   MRCLASS = {13D02 (05E40 13F65 13P10 14M05)},
  MRNUMBER = {4425287},
MRREVIEWER = {Jorge\ Neves},
       DOI = {10.1017/nmj.2021.1},
       URL = {https://doi.org/10.1017/nmj.2021.1},
}

@article {ass24,
    AUTHOR = {Ambhore, Siddhi Balu and Saha, Kamalesh and Sengupta,
              Indranath},
     TITLE = {The v-number of binomial edge ideals},
   JOURNAL = {Acta Math. Vietnam.},
  FJOURNAL = {Acta Mathematica Vietnamica},
    VOLUME = {49},
      YEAR = {2024},
    NUMBER = {4},
     PAGES = {611--628},
      ISSN = {0251-4184,2315-4144},
   MRCLASS = {13F20 (05E40 13F65)},
  MRNUMBER = {4834446},
MRREVIEWER = {Aryampilly\ V.\ Jayanthan},
       DOI = {10.1007/s40306-024-00540-w},
       URL = {https://doi.org/10.1007/s40306-024-00540-w},
}

@article {ss22,
    AUTHOR = {Saha, Kamalesh and Sengupta, Indranath},
     TITLE = {The v-number of monomial ideals},
   JOURNAL = {J. Algebraic Combin.},
  FJOURNAL = {Journal of Algebraic Combinatorics. An International Journal},
    VOLUME = {56},
      YEAR = {2022},
    NUMBER = {3},
     PAGES = {903--927},
      ISSN = {0925-9899,1572-9192},
   MRCLASS = {13F55 (05C70 05E40 13A15 13A70)},
  MRNUMBER = {4491066},
MRREVIEWER = {Somayeh\ Bandari},
       DOI = {10.1007/s10801-022-01137-y},
       URL = {https://doi.org/10.1007/s10801-022-01137-y},
}

@article {bht15,
    AUTHOR = {Beyarslan, Selvi and H\`a, Huy T\`ai and Trung, Tr\^an Nam},
     TITLE = {Regularity of powers of forests and cycles},
   JOURNAL = {J. Algebraic Combin.},
  FJOURNAL = {Journal of Algebraic Combinatorics. An International Journal},
    VOLUME = {42},
      YEAR = {2015},
    NUMBER = {4},
     PAGES = {1077--1095},
      ISSN = {0925-9899,1572-9192},
   MRCLASS = {13F20 (05C25 05C45 13D02)},
  MRNUMBER = {3417259},
MRREVIEWER = {John\ J.\ Watkins},
       DOI = {10.1007/s10801-015-0617-y},
       URL = {https://doi.org/10.1007/s10801-015-0617-y},
}

@article{liwski2025v,
  author  = {Liwski, E.},
  title   = {The {v}-number of binomial edge ideals: minimal cuts and cycle graphs},
  journal = {Collectanea Mathematica},
  year    = {2026},
  doi     = {10.1007/s13348-026-00511-4},
  url     = {https://doi.org/10.1007/s13348-026-00511-4}
}

@article {bms18,
    AUTHOR = {Bolognini, Davide and Macchia, Antonio and Strazzanti,
              Francesco},
     TITLE = {Binomial edge ideals of bipartite graphs},
   JOURNAL = {European J. Combin.},
  FJOURNAL = {European Journal of Combinatorics},
    VOLUME = {70},
      YEAR = {2018},
     PAGES = {1--25},
      ISSN = {0195-6698,1095-9971},
   MRCLASS = {05C25 (13C14 13F20)},
  MRNUMBER = {3779601},
MRREVIEWER = {Ali\ Reza\ Naghipour},
       DOI = {10.1016/j.ejc.2017.11.004},
       URL = {https://doi.org/10.1016/j.ejc.2017.11.004},
}

@article {Rauh13,
    AUTHOR = {Rauh, Johannes},
     TITLE = {Generalized binomial edge ideals},
   JOURNAL = {Adv. in Appl. Math.},
  FJOURNAL = {Advances in Applied Mathematics},
    VOLUME = {50},
      YEAR = {2013},
    NUMBER = {3},
     PAGES = {409--414},
      ISSN = {0196-8858,1090-2074},
   MRCLASS = {13F20 (05C25 13P10)},
  MRNUMBER = {3011436},
MRREVIEWER = {Siamak\ Yassemi},
       DOI = {10.1016/j.aam.2012.08.009},
       URL = {https://doi.org/10.1016/j.aam.2012.08.009},
}

@ARTICLE{JNR,
	AUTHOR = {Jayanthan, A. V. and Narayanan, N. and Raghavendra Rao, B. V.},
     TITLE = {Regularity of binomial edge ideals of certain block graphs},
   JOURNAL = {Proc. Indian Acad. Sci. Math. Sci.},
  FJOURNAL = {Indian Academy of Sciences. Proceedings. Mathematical Sciences},
    VOLUME = {129},
      YEAR = {2019},
    NUMBER = {3},
     PAGES = {Art. 36, 10},
      ISSN = {0253-4142},
   MRCLASS = {13D02 (05E40)},
  MRNUMBER = {3941158},
       DOI = {10.1007/s12044-019-0480-1},
       URL = {https://doi.org/10.1007/s12044-019-0480-1},
}

@article {cw19,
    AUTHOR = {Conca, Aldo and Welker, Volkmar},
     TITLE = {Lov\'asz-{S}aks-{S}chrijver ideals and coordinate sections of
              determinantal varieties},
   JOURNAL = {Algebra Number Theory},
  FJOURNAL = {Algebra \& Number Theory},
    VOLUME = {13},
      YEAR = {2019},
    NUMBER = {2},
     PAGES = {455--484},
      ISSN = {1937-0652,1944-7833},
   MRCLASS = {05E40 (05C62 13C40 13P10)},
  MRNUMBER = {3927052},
MRREVIEWER = {Dumitru\ Ioan\ Stamate},
       DOI = {10.2140/ant.2019.13.455},
       URL = {https://doi.org/10.2140/ant.2019.13.455},
}

@article {hmsw15,
    AUTHOR = {Herzog, J\"urgen and Macchia, Antonio and Saeedi Madani, Sara
              and Welker, Volkmar},
     TITLE = {On the ideal of orthogonal representations of a graph in
              {$\Bbb{R}^2$}},
   JOURNAL = {Adv. in Appl. Math.},
  FJOURNAL = {Advances in Applied Mathematics},
    VOLUME = {71},
      YEAR = {2015},
     PAGES = {146--173},
      ISSN = {0196-8858,1090-2074},
   MRCLASS = {05E40 (05C62 13F20)},
  MRNUMBER = {3406962},
MRREVIEWER = {Hassan\ Haghighi},
       DOI = {10.1016/j.aam.2015.09.009},
       URL = {https://doi.org/10.1016/j.aam.2015.09.009},
}

@misc{sz26,
      title={The {v}-number of generalized binomial edge ideals of some graphs}, 
      author={Yi-Huang Shen and Guangjun Zhu},
      year={arXiv:2603.29516, 2026},
      eprint={2603.29516},
      archivePrefix={arXiv},
      primaryClass={math.AC},
      url={https://arxiv.org/abs/2603.29516}, 
}

@article{zbMATH07352276,
 author = {Kumar, Arvind},
 title = {Regularity of parity binomial edge ideals},
 fjournal = {Proceedings of the American Mathematical Society},
 journal = {Proc. Am. Math. Soc.},
 issn = {0002-9939},
 volume = {149},
 number = {7},
 pages = {2727--2737},
 year = {2021},
 language = {English},
 doi = {10.1090/proc/15434},
 zbMATH = {7352276},
 Zbl = {1464.13011}
}

@article{liwski2025lov,
  title={Lov\'asz--Saks--Schrijver Ideals and the Irreducible Components of the Variety of Orthogonal Representations of a Graph},
  author={Liwski, Emiliano},
  journal={arXiv:2512.22954},
  year={2025}
}

@article{zbMATH05708656,
 author = {Hoa, Le Tuan and Tam, Nguyen Duc},
 title = {On some invariants of a mixed product of ideals},
 fjournal = {Archiv der Mathematik},
 journal = {Arch. Math.},
 issn = {0003-889X},
 volume = {94},
 number = {4},
 pages = {327--337},
 year = {2010},
 language = {English},
 doi = {10.1007/s00013-010-0112-6},
 zbMATH = {5708656},
 Zbl = {1191.13032}
}

@article {ar4,
    AUTHOR = {Kumar, Arvind},
     TITLE = {Lov\'asz-{S}aks-{S}chrijver ideals and parity binomial edge
              ideals of graphs},
   JOURNAL = {European J. Combin.},
  FJOURNAL = {European Journal of Combinatorics},
    VOLUME = {93},
      YEAR = {2021},
     PAGES = {Paper No. 103274, 19},
      ISSN = {0195-6698,1095-9971},
   MRCLASS = {05E40 (13A70)},
  MRNUMBER = {4186617},
       DOI = {10.1016/j.ejc.2020.103274},
       URL = {https://doi.org/10.1016/j.ejc.2020.103274},
}

@article {jks20,
    AUTHOR = {Jayanthan, A. V. and Kumar, Arvind and Sarkar, Rajib},
     TITLE = {Regularity of powers of quadratic sequences with applications
              to binomial ideals},
   JOURNAL = {J. Algebra},
  FJOURNAL = {Journal of Algebra},
    VOLUME = {564},
      YEAR = {2020},
     PAGES = {98--118},
      ISSN = {0021-8693,1090-266X},
   MRCLASS = {13D02 (05E40 13A70 13C13)},
  MRNUMBER = {4137693},
MRREVIEWER = {Jorge\ Neves},
       DOI = {10.1016/j.jalgebra.2020.08.004},
       URL = {https://doi.org/10.1016/j.jalgebra.2020.08.004},
}

@article{nkv26,
  author  = {Marie Amalore Nambi and Neeraj Kumar and Chitra Venugopal},
  title   = {{(Almost) Complete Intersection Lov{\'a}sz--Saks--Schrijver Ideals and the Regularity of Their Powers}},
  journal = {Journal of Algebra and Its Applications},
  volume  = {25},
  number  = {9},
  pages   = {2650093},
  year    = {2026},
  doi     = {10.1142/S0219498826500933},
  url     = {https://doi.org/10.1142/S0219498826500933}
}

@article{bms24,
  title={Asymptotic behaviour and stability index of v-numbers of graded ideals},
  author={Biswas, Prativa and Mandal, Mousumi and Saha, Kamalesh},
  journal={arXiv preprint arXiv:2402.16583},
  year={2024}
}

@article{JARAMILLO2021,
title = {The v-number of edge ideals},
journal = {Journal of Combinatorial Theory, Series A},
volume = {177},
pages = {105310},
year = {2021},
issn = {0097-3165},
doi = {https://doi.org/10.1016/j.jcta.2020.105310},
url = {https://www.sciencedirect.com/science/article/pii/S0097316520301023},
author = {Delio Jaramillo and Rafael H. Villarreal}
}

@article{fs25,
author = {Ficarra, Antonino and Sgroi, Emanuele},
title = {Asymptotic behavior of integer programming and the v-function of a graded filtration},
journal = {Journal of Algebra and Its Applications},
volume = {0},
number = {0},
pages = {2650236},
year = {0},
doi = {10.1142/S0219498826502361},

URL = {https://doi.org/10.1142/S0219498826502361
},
}

@article {cstpv20,
    AUTHOR = {Cooper, Susan M. and Seceleanu, Alexandra and Toh\u{a}neanu, S. O. and Pinto, Maria Vaz and Villarreal, Rafael H.},
     TITLE = {Generalized minimum distance functions and algebraic
              invariants of {G}eramita ideals},
   JOURNAL = {Adv. in Appl. Math.},
  FJOURNAL = {Advances in Applied Mathematics},
    VOLUME = {112},
      YEAR = {2020},
     PAGES = {101940, 34},
      ISSN = {0196-8858,1090-2074},
   MRCLASS = {13P25 (11T71 13C40 14G50 94B27)},
  MRNUMBER = {4011111},
MRREVIEWER = {C\'icero\ Carvalho},
       DOI = {10.1016/j.aam.2019.101940},
       URL = {https://doi.org/10.1016/j.aam.2019.101940},
}

@article {gkr93,
    AUTHOR = {Geramita, Anthony V. and Kreuzer, Martin and Robbiano,
              Lorenzo},
     TITLE = {Cayley-{B}acharach schemes and their canonical modules},
   JOURNAL = {Trans. Amer. Math. Soc.},
  FJOURNAL = {Transactions of the American Mathematical Society},
    VOLUME = {339},
      YEAR = {1993},
    NUMBER = {1},
     PAGES = {163--189},
      ISSN = {0002-9947,1088-6850},
   MRCLASS = {14M05 (13D40)},
  MRNUMBER = {1102886},
MRREVIEWER = {Luca\ Chiantini},
       DOI = {10.2307/2154213},
       URL = {https://doi.org/10.2307/2154213},
}

@article {c24,
    AUTHOR = {Conca, Aldo},
     TITLE = {A note on the {$v$}-invariant},
   JOURNAL = {Proc. Amer. Math. Soc.},
  FJOURNAL = {Proceedings of the American Mathematical Society},
    VOLUME = {152},
      YEAR = {2024},
    NUMBER = {6},
     PAGES = {2349--2351},
      ISSN = {0002-9939,1088-6826},
   MRCLASS = {13A30},
  MRNUMBER = {4741232},
MRREVIEWER = {Tony\ J.\ Puthenpurakal},
       DOI = {10.1090/proc/16767},
       URL = {https://doi.org/10.1090/proc/16767},
}

@misc{jls26,
      title={Private neighbors, perfect codes and their relation with the $\mathtt{v}$-number of closed neighborhood ideals}, 
      author={Delio Jaramillo-Velez and Hiram H. López and Rodrigo San-José},
      year={arxiv:2603.28247, 2026},
      eprint={2603.28247},
      archivePrefix={arXiv},
      primaryClass={math.AC},
      url={https://arxiv.org/abs/2603.28247}, 
}

@article {nk24,
    AUTHOR = {Amalore Nambi, Marie and Kumar, Neeraj},
     TITLE = {Regularity of powers of {$d$}-sequence (parity) binomial edge
              ideals of unicycle graphs},
   JOURNAL = {Comm. Algebra},
  FJOURNAL = {Communications in Algebra},
    VOLUME = {52},
      YEAR = {2024},
    NUMBER = {6},
     PAGES = {2598--2615},
      ISSN = {0092-7872,1532-4125},
   MRCLASS = {13F65 (05E40)},
  MRNUMBER = {4728966},
MRREVIEWER = {Luca\ Amata},
       DOI = {10.1080/00927872.2024.2302101},
       URL = {https://doi.org/10.1080/00927872.2024.2302101},
}

@article {sz23,
    AUTHOR = {Shen, Yi-Huang and Zhu, Guangjun},
     TITLE = {Regularity of powers of (parity) binomial edge ideals},
   JOURNAL = {J. Algebraic Combin.},
  FJOURNAL = {Journal of Algebraic Combinatorics. An International Journal},
    VOLUME = {57},
      YEAR = {2023},
    NUMBER = {1},
     PAGES = {75--100},
      ISSN = {0925-9899,1572-9192},
   MRCLASS = {13F65 (05E40 13C13 13C15 13D02 13F20)},
  MRNUMBER = {4544259},
MRREVIEWER = {Kazunori\ Matsuda},
       DOI = {10.1007/s10801-022-01163-w},
       URL = {https://doi.org/10.1007/s10801-022-01163-w},
}

@misc{cj26,
  author       = {Trung Chau and A. V. Jayanthan},
  title        = {The {v}-numbers of permanental ideals},
  year         = {arXiv:2605.11621, 2026},
  eprint       = {2605.11621},
  archivePrefix= {arXiv},
  primaryClass = {math.AC},
}
\end{document}